\documentclass[reqno,12pt]{amsart}
\usepackage[a4paper,hmargin=2.5cm,vmargin=2.5cm]{geometry}
\usepackage[utf8]{inputenc}
\usepackage[english]{babel}
\RequirePackage{xr-hyper}[6.00]
\usepackage{hyperref}
\usepackage{amsthm}

\newtheorem{theorem}{Theorem}[section]

\newtheorem{lemma}[theorem]{Lemma}
\newtheorem{conjecture}[theorem]{Conjecture}
\newtheorem{proposition}[theorem]{Proposition}

\theoremstyle{remark}
\newtheorem*{remark}{Remark}
\newtheorem{problem}[theorem]{Problem}

\theoremstyle{definition}
\newtheorem{definition}[theorem]{Definition}

\usepackage{mathrsfs}

\RequirePackage{algorithm}[0.1]

\RequirePackage[capitalize,nameinlink]{cleveref}[0.19]
\crefname{section}{section}{sections}
\crefname{subsection}{subsection}{subsections}
\Crefname{section}{Section}{Sections}
\Crefname{subsection}{Subsection}{Subsections}
\Crefname{figure}{Figure}{Figures}
\crefname{problem}{Problem}{Problem}
\crefname{hypothesis}{Hypothesis}{Hypotheses}
\crefname{conjecture}{Conjecture}{Conjecture}
\crefname{proposition}{Proposition}{Proposition}

\usepackage{amsopn}
\usepackage{algorithmic}
\Crefname{ALC@unique}{Line}{Lines}

\usepackage{lipsum}
\usepackage{amsfonts}
\usepackage{graphicx}
\usepackage{epstopdf}
\usepackage{algorithmic}

\newcommand{\eps}{\varepsilon}

\newcommand{\supp}{\mathrm{supp}}
\newcommand{\const}{\mathrm{const}}

\title{An explicit solution for a multimarginal mass transportation problem}

\author{Nikita~A.~Gladkov}
\address{National Research University Higher School of Economics, Moscow, Russia}
\email{gladkovna@gmail.com}

\author{Alexander~P.~Zimin}
\address{National Research University Higher School of Economics, Moscow, Russia}
\email{alekszm@gmail.com}


\thanks{
	This work has been funded by the Russian Academic Excellence Project '5-100'.
}

\begin{document}

\begin{abstract}
We construct an explicit solution for the multimarginal transportation problem on the unit cube $[0, 1]^3$ with the cost function $xyz$ and one-dimensional uniform projections.
We show that the primal problem is concentrated on a set with non-constant local dimension and admits many solutions, whereas the solution to the corresponding dual problem is unique (up to addition of constants).
\end{abstract}

\maketitle

\section{Introduction}
\subsection{Notation}
Assume we are given $n$ polish spaces $X_1$, $X_2$, \dots, $X_n$, equipped with probability measures $\mu_i$ on $X_i$ and a cost function $c: X_1 \times \dots \times X_n \to \mathbb{R}$. 

In multimarginal Monge-Kantorovich problem (called primal problem throughout this paper) we seek to minimize $$\int_{X_1 \times \dots \times X_n}
c(x_1, x_2, ..., x_n)~d\mu(x_1, x_2, \dots, x_n)$$
over the set $\Pi(\mu_1, \mu_2, \dots, \mu_n)$ of positive joint measures $\mu$ on the product space
$X_1 \times \dots \times X_n$ whose marginals are the $\mu_i$. See \cite{Villani, bogachev} for an account in the optimal transportation problem with two marginals and \cite{brendan_pass}.

With the primal problem people also consider the dual problem. Under conditions above we are concerned with the supremum of $$\sum_{i = 1}^n \int_{X_i}f_i(x_i)~d\mu_i$$
where supremum is taken over all sets of functions $\{f_i\}$ such that $\sum_{i = 1}^n f_i(x_i) \le c(x_1, \dots, x_n)$ for any $x_i \in X_i$.

It is easy to show that minimum in primal problem is less or equal to the supremum in dual problem. Under some conditions it is true that this numbers are equal \cite{Villani, brendan_pass, Kellerer}. 

We do not need a full power of duality here. This paper relies on the following easy fact.

\begin{lemma}[Complementary Slackness Condition]\label{lem:slackness_conditions}
Let $\mu \in \Pi(\mu_1, \dots, \mu_n)$ be a joint measure and $f_1, f_2, \dots, f_n$ be a tuple of functions such that $\sum_{i = 1}^n f_i(x_i) \le c(x_1, \dots, x_n)$. 
If there is a set $M \subset X_1 \times X_2 \times \dots \times X_n$ such that on $M$ one has $\sum_{i = 1}^n f_i(x_i) = c(x_1, \dots, x_n)$ with the additional property $\mu(M)=1$, then $\mu$ is a primal solution and $f_i$ is a dual solution.
\end{lemma}

The aim of this paper is to describe an example of explicit solution to the mass transportation problem on $[0, 1]^3$ ($X_1 = X_2 = X_3 = [0, 1]$) with one-dimensional Lebesgue measure projections and the cost function $c(x, y, z) = xyz$. In this paper we call the measures on $[0, 1]^3$ with Lebesgue projections onto the axes $(3, 1)-$stochastic measures.

In fact, we will construct the primal and dual solutions for any cost function $c(x, y, z) = C(xyz)$ for some continuously differentiable function $C:[0, 1] \to \mathbb{R}$ such that the function $tC'(t)$
strictly increases on the segment $[0, 1]$.

\subsection{Motivation}
Our problem appears to be the simplest generalization of the classical Monge--Kantorovich problem with one-dimensional marginals and quadratic cost function. It seems to be never considered in the literature, though other generalizations mentioned in \cref{subsection:relatedproblems} received some attention.
Note that the particular cost function $(x-y)^2$ (equivalently $- xy$) is mostly used in the classical Monge--Kantorovich theory. A natural replacement of $-xy$ for the case of three variables is $-xyz$. For the cost function $-xyz$ the solution to the primal problem with the same marginals admits a simple structure: it is concentrated on the main diagonal of $[0, 1]^3$ (this can be viewed as a ``continuous rearrangement inequality'' or ``Hardy-Littlewood inequality''). Unlike this, solutions for $xyz$ are non-trivial, that is why we are interested in the cost function $xyz$.

\subsection{Main results}
In this paper we construct the set $M$ which is $c-$monotone for the cost function $c(x, y, z) = xyz$. The set $M$ is the union of three segments and one 2-dimensional part as below: 
\begin{align*}
    &M_x = \{(t, 1 - 2t, 1 - 2t) \mid 0 \le t \le l\},\\
    &M_y = \{(1 - 2t, t, 1 - 2t) \mid 0 \le t \le l\},\\
    &M_z = \{(1 - 2t, 1 - 2t, t) \mid 0 \le t \le l\},\\
    &M_2 = \{(x, y, z) \mid l \le x, y, z \le r = 1 - 2l, xyz = lr^2\}, \\
    &M = M_x \cup M_y \cup M_z \cup M_2,
\end{align*}
where $l \approx 0.0945$, $r \approx 0.8119$ are some transcendent constants.

\begin{figure}[htbp]
    \centering
    \includegraphics[width=0.7\textwidth]{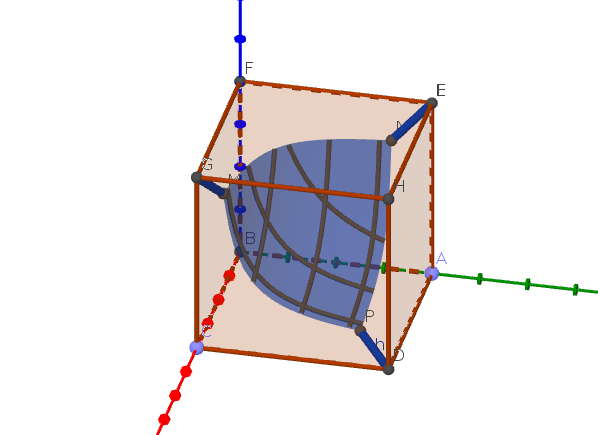}
    \caption{Set $M$}
\end{figure}

Initially, we got an explicit construction of this set from heuristic considerations (see \cref{section:heuristic}). In \cref{section:primal_construction} we see that the integral $\int xyz~d\mu$ is the same for any $(3, 1)-$stochastic measure $\mu$ such that $\supp(\mu) \subset M$ (see \cref{prop:same_integral}). After that we explicitly construct a $(3, 1)-$stochastic measure $\pi$ concentrated on the set $M$ (see the proof of \cref{thm:osexists}). The proof contains nontrivial construction and technical computations. The constructed measure is the primal solution of the related transport problem.

To prove that the measure $\pi$ is the primal solution in \cref{section:dual} we solve a related dual problem for a cost function $c(x, y, z) = C(xyz)$. Our proof works for $C:[0, 1] \to \mathbb{R}$ such that the function $tC'(t)$
strictly increases on the segment $[0, 1]$. Naturally that means $C(xyz)=\widehat{C}(\ln x + \ln y + \ln z)$ where $\widehat{C}$ is a bounded continuously differentiable convex function on $(-\infty, 0]$.

The following theorem gives an explicit construction for the dual solution (see \cref{thm:primal_dual}). Thus, together with \cref{thm:osexists} it gives a characterization of both primal and dual solution.

\begin{theorem}[Main result]\label{thm:intro_dual}
Suppose that $c(x, y, z) = C(xyz)$ for some continuously differentiable function $C:[0, 1] \to \mathbb{R}$ and the function $tC'(t)$
strictly increases on the segment $[0, 1]$. Set: 
$$\hat{f}(s) = \int_{0}^{s}\lambda(t)C'(t\lambda(t))~dt,$$ where the function $\lambda$ is as in \cref{def:lambda_definition}. Then for any constants $C_x$, $C_y$, $C_z$ such that $$C_x + C_y + C_z = C(0) - 2\int_{0}^{1}\lambda(t)C'(t\lambda(t))~dt$$ the following inequality holds 
$$(\hat{f}(x) + C_x) + (\hat{f}(y) + C_y) + (\hat{f}(z) + C_z) \le c(x, y, z)$$
with equality on $M$.

Using the complementary slackness conditions (see \cref{lem:slackness_conditions}) we conclude that for any cost function $C(xyz)$, for which the conditions above are satisfied, any $(3, 1)-$stochastic measure $\pi$ with $\supp(\pi) \subset M$ is a primal solution for a multimarginal mass transportation problem and functions $\hat{f}$ defined in \cref{thm:intro_dual} is a dual solution.
\end{theorem}

In \cref{subsection:explicit_dual_xyz} the explicit form of the dual solution for the cost function $c(x, y, z) = xyz$ is specified. It has the following form (see \cref{prop:explicit_dual_xyz}):
\begin{align*}
    &\hat{f}(x) = \begin{cases}
    c \ln l - {1 \over 3}(c\ln c - c) + {1 \over 6}((2x - 1)^3 - (2l - 1)^3), &\text{if } 0 \le x \le l,\\
c\ln x - {1 \over 3}(c\ln c - c), &\text{if } l \le x \le r,\\ 
c\ln r - {1 \over 3}(c\ln c - c) + {1 \over 4}(x^2 - r^2) - {1 \over 6}(x^3 - r^3), &\text{if } r \le x \le 1,
    \end{cases}\\
    &f(t) = g(t) = h(t) = \hat{f}(t),
\end{align*}
for constants $l, r, c$.

In \cref{section:uniqueness} we prove that for any cost function $c(x, y, z) = C(xyz)$ dual solution is unique up to adding constants and measure zero.

Structural results (see \cite{Pass12, brendan_pass}) allow us to estimate the local dimension $d$ of $M$. We apply this results in \cref{section:inertion} to see that $d$ is bounded by $2$. The dimension of the support is important for computations and was studied in details in \cite{Friesecke}. It is interesting that the local dimension of $M$ is not constant as $M$ admits one-dimensional parts and a two-dimensional part. 

This two-dimensional part is a source of non-uniqueness for the primal problem. After the logarithmic change of coordinates cost function $C(xyz)$ becomes convex in sum of coordinates, Lebesgue measure on axis becomes an exponential distribution and two-dimensional part of $M$ becomes a triangle on a plane $x+y+z=\const$. This resembles the situation in \cite[Lemma 4.3]{dim-ger-nen} where the authors consider the multimarginal problem with the same cost function and Lebesgue marginals. They prove that the plan is optimal if and only if it is concentrated on a plane $x+y+z=\const$.

The cost function $xyz$ violates the standard uniqueness assumption, the so-called twist condition (see \cite{KP, Pass11, brendan_pass}). The primal problem admits many solutions. In particular, we show that there exist solutions which are singular with respect to the Hausdorff measure on $M$.
We also propose the following 
\begin{conjecture}\label{conj:hausdorff}
There exists a solution which is concentrated on a set which has Hausdorff dimension less than $2$.
\end{conjecture}

This conjecture is motivated by \cite[Theorem 4.6]{dim-ger-nen} where the authors construct a primal solution with a fractal support.

\subsection{Related problems}\label{subsection:relatedproblems}
Our example contributes to the list of several known explicit examples and to the list of cost functions where the structure of solutions is investigated in details.
Here are some other examples.
\begin{enumerate}
\item Cost function $$-\sum_{i \ne j} x_i x_j .$$
This cost function is related to the geodesic barycenter problem (see \cite{Carlier, AguehC}).
\item
Determinantal cost \cite{CarNaz}.
\item
Coulomb cost \cite{CFK} (see \cite{CPM} for generalizations). The motivation for this problem comes from mathematical physics.
\item
$\min(x_1, \ldots, x_n)$ (more generally, minumum of affine functions) \cite{KL}.
\item 
Convex function of $x_1+ \ldots + x_n$ (see \cite{dim-ger-nen}).
\end{enumerate}
Some other examples can be found in \cite{brendan_pass}.	

Also, our problem is closely related to $(3, 2)$ problem, studied in \cite{main_work}. In particular, our example can be considered as a solution to the primal $(3, 2)$-problem with the same cost function $xyz$ and the corresponding 2-dimensional projections.
In the $(3, 2)$-problem we consider a modification of the transportation problem. Namely, we deal with the space of measures with fixed projections onto $$X_1 \times X_2, \ X_2 \times X_3, \ X_1 \times X_3.$$
The main result of \cite{main_work} describes a solution to the $(3, 2)$-problem on $[0, 1]^3$ with the cost function $xyz$ ($-xyz$) and two-dimensional Lebesgue measure projections. It turns out that in strong contrast with the classical transportation problem the solution is supported by the fractal set (Sierpi\'nski tetrahedron) $$z = x \oplus y,$$ where $\oplus$ is bitwise addition.
Let us also mention another related important modification: Monge--Kantorovich problem with linear constraints, which has been introduced and studied in \cite{zaev}. 

\section{An heuristic description of $M$}\label{section:heuristic}

In this subsection we collect some informal observations related to our main construction.
In particular, we briefly analyze the cyclical monotonicity property of the support set of our primal solution and describe how to approach the problem numerically. 

Let $M$ be a full measure set for the primal solution. Since all the marginals and the cost function are symmetric with respect to the coordinate axes interchange, we may assume without loss of generality  that $M$ is also symmetric in this sense.

The set $M$  can be chosen to be $c-$cyclically monotone. This is well known for two marginals, for many marginals we refer to the work \cite{griessler}. In particular that means that for any $(x_1, y_1, z_1), (x_2, y_2, z_2) \in M$ one has
\begin{align}\label{eq:c_monotonicity}
\begin{split}
    c(x_1, y_1, z_1) + c(x_2, y_2, z_2) \le c(x_2, y_1, z_1) + c(x_1, y_2, z_2),\\
    c(x_1, y_1, z_1) + c(x_2, y_2, z_2) \le c(x_1, y_2, z_1) + c(x_2, y_1, z_2),\\
    c(x_1, y_1, z_1) + c(x_2, y_2, z_2) \le c(x_1, y_1, z_2) + c(x_2, y_2, z_1).
\end{split}
\end{align}

The \cref{alg:primal_solution_approximation_general} constructing an approximation to a primal solution is based on the inequality above.

\begin{algorithm}
\caption{Primal solution approximation (general version)}
\label{alg:primal_solution_approximation_general}
\begin{algorithmic}[1]
\STATE{Generate three samples: $x_1, x_2, \dots, x_n$ from $\mu_1$, $y_1, y_2, \dots, y_n$ from $\mu_2$, $z_1, z_2, \dots, z_n$ from $\mu_3$; $n$ is a parameter, $\mu_i$ are the marginals in the primal problem.}
\STATE{Define $S := \{(x_k, y_k, z_k) \text{ for } 1 \le k \le n\}$.}
\WHILE{$S$ doesn't satisfy \cref{eq:c_monotonicity}}
\STATE{Take two points $(a_1, b_1, c_1)$ and $(a_2, b_2, c_2)$ from $S$.}
\IF{$c(a_1, b_1, c_1) + c(a_2, b_2, c_2) > c(a_2, b_1, c_1) + c(a_1, b_2, c_2)$}
    \STATE{replace $(a_1, b_1, c_1)$ and $(a_2, b_2, c_2)$ with $(a_2, b_1, c_1)$ and $(a_1, b_2, c_2)$ is $S$}
\ELSIF{$c(a_1, b_1, c_1) + c(a_2, b_2, c_2) > c(a_1, b_2, c_1) + c(a_2, b_1, c_2)$}
    \STATE{replace $(a_1, b_1, c_1)$ and $(a_2, b_2, c_2)$ with $(a_1, b_2, c_1)$ and $(a_2, b_1, c_2)$ is $S$}
\ELSIF{$c(a_1, b_1, c_1) + c(a_2, b_2, c_2) > c(a_1, b_1, c_2) + c(a_2, b_2, c_1)$}
    \STATE{replace $(a_1, b_1, c_1)$ and $(a_2, b_2, c_2)$ with $(a_1, b_1, c_2)$ and $(a_2, b_2, c_1)$ is $S$}
\ENDIF
\ENDWHILE
\STATE{$S$ is an approximation of the primal solution.}
\end{algorithmic}
\end{algorithm}

In our case $c(x, y, z) = xyz$, so
\begin{align*}
    &x_1y_1z_1 + x_2y_2z_2 \le x_1y_1z_2 + x_2y_2z_1,\\
    &(x_1y_1 - x_2y_2)(z_1 - z_2) \le 0.
\end{align*}
It follows that if $(x_1, y_1, z_1), (x_2, y_2, z_2) \in M$ then
\begin{align}\label{cond:weak_monotonicity}
\begin{split}
&z_1 < z_2 \Rightarrow x_1y_1 \ge x_2y_2 \text{ and by the symmetry}\\
&y_1 < y_2 \Rightarrow x_1z_1 \ge x_2z_2,\\
&x_1 < x_2 \Rightarrow y_1z_1 \ge y_2z_2.
\end{split}
\end{align}

\begin{algorithm}
	\setcounter{ALC@unique}{0}
\caption{Primal solution approximation (faster version)}
\label{alg:primal_solution_approximation_faster}
\begin{algorithmic}[1]
\STATE{Generate three samples: $x_1, x_2, \dots, x_n$ from $\mu_1$, $y_1, y_2, \dots, y_n$ from $\mu_2$, $z_1, z_2, \dots, z_n$ from $\mu_3$; $n$ is a parameter, $\mu_i$ are the marginals in the primal problem.}
\STATE{Define $S := [(x_k, y_k, z_k) \text{ for } 1 \le k \le n]$.}
\WHILE{$S$ doesn't satisfy \cref{cond:weak_monotonicity}}
\STATE{Sort $S$ by the first coordinate in the ascending order. Denote by $(a_k, b_k, c_k)$ the $k-$th item of $S$ after sorting, $1 \le k \le n$.}\label{alg:step_sort_1}
\STATE{Update $S := [(a_k, b_{\sigma(k)}, c_{\sigma(k)}) \text{ for } 1 \le k \le n]$ where $\sigma \in S_n$ and the sequence $b_{\sigma(k)}c_{\sigma(k)}$ is descending.}\label{alg:step_sort_2}
\STATE{Repeat \cref{alg:step_sort_1} and \cref{alg:step_sort_2} for the second and third coordinates.}
\ENDWHILE
\STATE{$S$ is an approximation of the primal solution.}
\end{algorithmic}
\end{algorithm}

\begin{figure}[htbp]
    \centering
    \includegraphics[width=0.7\textwidth]{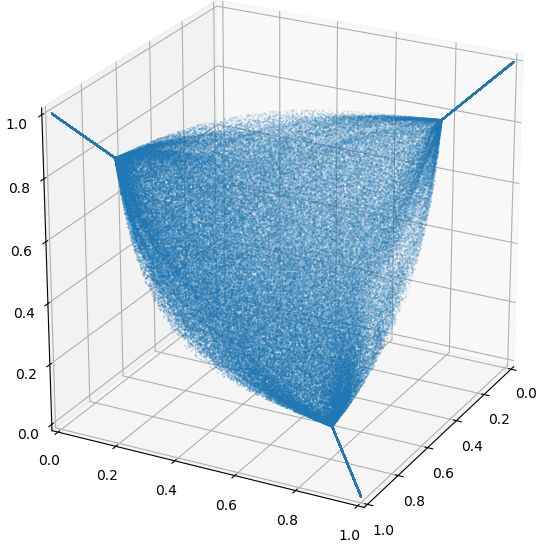}
    \caption{The final set $S$ for $n = 200000$.}
    \label{fig:sorted_M}
\end{figure}

This allows us to improve the performance of \cref{alg:primal_solution_approximation_general} using sortings. That leads us to a much faster version, namely \cref{alg:primal_solution_approximation_faster}. We were able to run an \cref{alg:primal_solution_approximation_faster} for $n=2 \times 10^5$. 

Despite the fact that this algorithm does not necessarily converges to the solution
for all admissible data, our numerical experiments demonstrate that the algorithm works well in
many cases. A proof of convergence for a suitable set of admissible data must be investigated.

On \cref{fig:sorted_M} it is shown a scatter plot of $S$ after the completion of the algorithm. As one can see on this graph, the set $M$ consists of four parts. There exist real values $0 < l < r < 1$ such that if $l \le x, y, z \le r$ and $(x, y, z) \in M$ then $(x, y, z)$ lies on 2-dimensional part $M_2$. If $0 \le x \le l$ and $(x, y, z) \in M$ then $(x, y, z)$ lies on a 1-dimensional curve $M_x = (p(t), p_y(t), p_z(t))$, $0 \le t \le 1$. By the virtue of the symmetry $p_y(t) = p_z(t) = q(t)$. Also $q(0) = 1$, $q(1) = r$ and $q(t)$ strictly decrease; $p(0) = 0$, $p(1) = l$ and $p(t)$ strictly increase.

By the virtue of the symmetry, if $(x, y, z) \in M$ and $0 \le y \le l$, then this point lies on a curve $M_y = (q(t), p(t), q(t))$, and if $0 \le z \le l$ then $(x, y, z)$ lies on a curve $M_z = (q(t), q(t), p(t))$.

Let $\nu$ be a primal solution and $\nu_x$, $\nu_y$, $\nu_z$ be the restrictions of $\nu$ to $M_x, M_y$ and $M_z$ accordingly. Suppose $F_x(a) = \nu_x(\{(p(t), q(t), q(t)) \mid 0 \le t \le a\})$. Define $F_y$ and $F_z$ in a similar way. By the virtue of the symmetry we can assume that $F_x(a) = F_y(a) = F_z(a) = F(a)$ for any $0 \le a \le 1$.

For any $0 \le a \le 1$ one has 
\begin{multline*}
    \nu(\{0 \le x \le p(a)\}) = \nu_x(\{(p(t), q(t), q(t)) \mid 0 \le t \le a\}) = F_x(a) \\
    = {1 \over 2}(F_y(a) + F_z(a)) = {1 \over 2}\nu(\{q(a) \le x \le 1\}).
\end{multline*}
Since all the marginals are Lebesgue measures on the segments $[0, 1]$, one has $2p(a) = 1 - q(a)$ for any $0 \le a \le 1$. 

Thus $$M_x = (t, 1 - 2t, 1-2t)$$
$$   M_y = (1 - 2t, t, 1 - 2t)$$ 
and $$M_z = (1 - 2t, 1-2t, t),$$ $0 \le t \le l$; $r = 1 - 2l$. That means that 1-dimensional parts of the set $M$ are segments.

The set $M_x$ is $c-$cyclically monotone. In particular, if $1 - 2t_1 < 1 - 2t_2$, then $t_1(1 - 2t_1) \ge t_2(1 - 2t_2)$ or, equivalently, the function $t(1 - 2t)$ increases on the set $[0, l]$. The derivative of this function is $1 - 4t \ge 0$ for any $0 \le t \le l$. That means that $0 < l \le {1 \over 4}$.

Let us describe the set $M_2$. As we know from the general duality theory, there exist functions $$f, g, h: [l, r] \to \mathbb{R}$$ satisfying  $f(x) + g(y) + h(z) \le xyz$ and the equality holds provided $(x, y, z) \in M_2$. Again by symmetry we can assume that $f(x) = g(x) = h(x)$ for any $l \le x \le r$. 

Suppose that $f$ is continuously differentiable. Let $(x, y, z) \in M_2$ and $l \le x, y, z \le r$. Then that point is an inner maximum point of the function $$F(x, y, z) = f(x) + f(y) + f(z) - xyz.$$ That means that 

$$\nabla F = (f'(x) - yz, f'(y) - xz, f'(z) - xy)^T = \vec{0}.$$

So if $(x, y, z) \in M_2$ then $xf'(x) = yf'(y) = zf'(z) = xyz$. From \cref{fig:sorted_M} we see that if we fix $x=l$, then for any $l \le y \le r$ there exists $l \le z \le r$ such that $(x, y, z) \in M_2$. Then the function $tf'(t)$ is equal to a constant $C=lf'(l)$ for any $l \le t \le r$. In this case if $(x, y, z) \in M_2$ then $xyz = xf'(x) = C$. Since $(l, r, r) \in M_2$ the constant $C$ is equal to $lr^2$.

\section{Solving the primal problem}\label{section:primal_construction}
Summarizing the facts about the set $M$, which supports the primal solutions, we realize that one can try to find $M$
in the following form:
\begin{align*}
    &M_x = \{(t, 1 - 2t, 1 - 2t) \mid 0 \le t \le l\},\\
    &M_y = \{(1 - 2t, t, 1 - 2t) \mid 0 \le t \le l\},\\
    &M_z = \{(1 - 2t, 1 - 2t, t) \mid 0 \le t \le l\},\\
    &M_2 = \{(x, y, z) \mid l \le x, y, z \le r = 1 - 2l, xyz = lr^2\},
\end{align*}
where $l$ is an unknown parameter; $0 \le l < {1 \over 4}$.

\begin{proposition}\label{prop:same_integral}
An integral $\int xyz~d\nu(x, y, z)$ is the same for any probability measure  $\nu$ such that $\Pr_x(\nu) = \Pr_y(\nu) = \Pr_z(\nu) = \lambda$ where $\lambda$ is the Lebesgue measure on the segment $[0, 1]$ and $\supp(\nu) \subset M$.
\end{proposition}
\begin{proof}
Let $\nu_x$, $\nu_y$, $\nu_z$ and $\nu_2$ be restrictions of $\nu$ to $M_x$, $M_y$, $M_z$ and $M_2$ respectively. Since the projection of $\nu_x$ on the first marginal is a restriction of $\lambda$ to the segment $[0, l]$, one has
$$
\int_{M_x} xyz~d\nu_x(x, y, z) = \int x(1 - 2x)(1-2x)~d\nu_x(x, y, z) = \int_0^lx(1-2x)^2~dx
$$
Similarly
$$
\int_{M_y} xyz~d\nu_y(x, y, z) = \int_{M_z} xyz~d\nu_z(x, y, z) = \int_0^lx(1-2x)^2~dx.
$$

Finally, the projection of $\nu_2$ on the first marginal is a restriction of $\lambda$ to the segment $[l, r]$. So
$$
\int_{M_2}xyz~d\nu_2 = lr^2\cdot \nu_2(\{l \le x \le r\}) = lr^2(r - l).
$$

Consequently $\int xyz~d\nu(x, y, z) = 3\int_0^lx(1 - 2x)^2~dx + lr^2(r - l)$ and this integral does not depend on $\nu$.
\end{proof}

So we only have to find any measure with desired projections such that its support is contained in $M$. In \cref{thm:primal_dual} we find an appropriate triple of functions and by \cref{lem:slackness_conditions} we rigorously prove that any $(3, 1)$-stochastic measure on $M$ is indeed a primal solution.

First, we define a measure on the three one-dimensional segments. Let $L = \sqrt{l^2 + 2(1 - r)^2}$ be the lengths of these segments. We set on every segment a uniform measure with density $\frac{l}{L}$.
Clearly, projections of two segments coincide with $[r, 1]$, the densities are equal to $\frac{L}{1 - r} \cdot \frac{l}{L} = \frac{1}{2}$. Their sum is the Lebesgue measure on $[r, 1]$.
The projection of the third interval is a measure on $[0, l]$, its density equals $\frac{L}{l} \cdot \frac{l}{L} = 1$.

After this, it remains to determine the measure on the remaining two-dimensional set such that its projection on each of the axes is uniform.

Let us make the following change of coordinates: $$u := {\ln x - \ln l \over \ln r - \ln l}, v := {\ln y - \ln l \over \ln r - \ln l}, w := {\ln z - \ln l \over \ln r - \ln l}.$$
Two-dimensional set
$$xyz = c, l \le x, y, z \le r$$ admits the following parametrization: $$u + v + w = 2, 0 \le u, v, w \le 1.$$
One has the following relations:
\begin{align*}
    &dx = de^{u(\ln r - \ln l) + \ln l} = l\ln \left(r \over l \right) \left({r \over l}\right)^udu = l\ln(\alpha) \alpha^u du,\\
    &dy = l\ln(\alpha) \alpha^v dv,\\
    &dz = l\ln(\alpha) \alpha^w dw,
\end{align*}
where $\alpha = \frac{r}{l}$.

Clearly, the problem is reduced to the following problem: find a measure on the triangle $u + v + w = 2, 0 \le u, v, w \le 1$ with exponential projections onto the axes.

\subsection{Necessary conditions for existence of a measure on the triangle with given projections}

\

One can put the problem into a more general setting. 
When there exists a measure $\mu$ on the triangle $$\Delta = \{x + y + z = 2,~0 \le x, y, z \le 1\}$$ with given projections $\mu_x$, $\mu_y$, $\mu_z$?

In what follows we are only interested in the case $\mu_x = \mu_y = \mu_z = \pi$. A necessary condition is given in the following lemma.

\begin{lemma} \label{inequality}
Let function $f : [0, 1] \to \mathbb{R}$ satisfy $f(x) + f(y) + f(z) \le 0$ for $x + y + z = 2$ and there exist a measure $\mu$ on $\Delta$, whose projections 
onto the axes are equal to $\pi$. Then $\int_0^1 f(x) d\pi \le 0$.
\end{lemma}

\begin{proof}
We compute $\int_\Delta (f(x) + f(y) + f(z))~d\mu$. On the one hand, it is nonpositive, since at each point $f(x) + f(y) + f(z) \le 0$. On the other hand,
$$
\int_\Delta (f(x) + f(y) + f(z))d\mu = 3\int_0^1 f(x)d\pi(x) \le 0.
$$

\end{proof}

In particular, for the function $f(x)=x-\frac{2}{3}$ one has $f(x) + f(y) + f(z) = 0$ for $x + y + z = 2$. So we get

\begin{equation} \label{main_traingle_equality}
\int_0^1 \left(x - {2 \over 3}\right) d\pi(x) = 0.
\end{equation}

Check this for $d\pi = \alpha^xdx$:
$$
\int_0^1 \left(x - {2 \over 3}\right)d\pi = \int_0^1 \left(x - {2 \over 3}\right)\alpha^xdx = {{\alpha(\ln\alpha - 3) + 3 + 2\ln\alpha} \over 3\ln^2\alpha}
$$

Thus, $\alpha$ must satisfy
\begin{equation} \label{alpha_equation}
{\alpha(\ln\alpha - 3) + 3 + 2\ln\alpha} = 0.
\end{equation}
Apply the relation $\alpha = {1 - 2l \over l}$:

$$
\alpha(\ln\alpha - 3) + 3 + 2\ln\alpha 
= {\ln(1 - 2l) - \ln l - 3 + 9l \over l} = 0.
$$

\begin{figure}
    \centering
    \includegraphics[width=\textwidth]{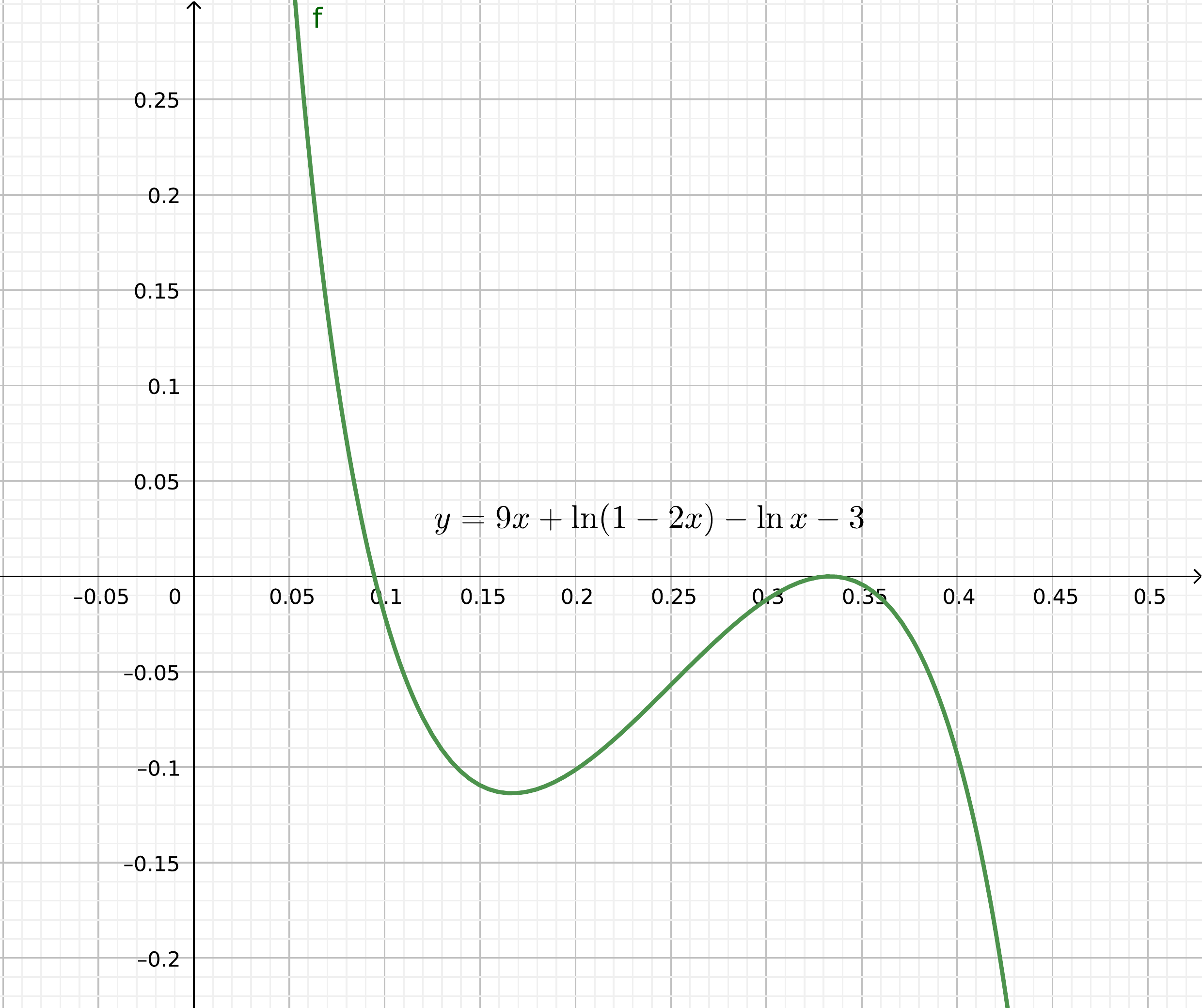} 
    \caption{A graph of a function $y(x) = 9x + \ln(1 - 2x) - \ln x - 3$}
    \label{fig:plot_l}
\end{figure}

It is seen from \cref{fig:plot_l} that the function $\ln(1 - 2l) - \ln l - 3 + 9l$ has exactly one root lying in the interval $\left(0, {1 \over 4}\right)$, namely $l\approx 0.0945$. So $r \approx 0.8109$ and $\alpha \approx 8.577$.

Let us prove that there is a unique root of $h$ lying inside $\left(0, {1 \over 3}\right)$. To this end we find the derivative of $h(l) = 9l + \ln(1 - 2l) - \ln l - 3$ and show it is negative for $l < \frac{1}{6}$ and positive for $\frac{1}{6} < l < \frac{1}{3}$. Indeed,
$$
h'(l) = (9l + \ln(1 - 2l) - \ln l - 3)' = -\frac{(3l - 1)(6l - 1)}{l(1-2l)},
$$
and it is easy to check the signs.

For $l\to+0$ one has $$h(l) \to \infty.$$ For $l = {1 \over 6}$ there holds
$$
h\left(\frac{1}{6}\right) = 2\ln{2} - {3 \over 2} < 0,
$$
since $\ln 2 \approx 0.69 < {3 \over 4}$.
For $l = {1 \over 3}$ there holds
$$
h\left(\frac{1}{3}\right)= 3 + \ln\left(1 - {2 \over 3}\right) - \ln\left(1 \over 3\right) - 3 = 0.
$$

It follows that on the interval {$\left(0, {1 \over 4}\right)$} function $h(l)$ has exactly one root {and this root is less than $1 \over 6$.}

Assumption of \cref{inequality} is satisfied for the following broad class of functions. 
\begin{lemma}
Let $f(x): [0, 1] \to \mathbb{R}$ be convex on $\left[0, {2 \over 3}\right]$ and $f(2x) + 2f(1 - x) = 0$ for $0 \le x \le {1 \over 3}$. Then $f(x) + f(y) + f(z) \le 0$ for $x + y + z = 2$.
\end{lemma}

\begin{proof}
Assume that $f(x) + f(y) + f(z) > 0$ for some $x$, $y$ and $z$ satisfying $x + y + z = 2$. 
Let among $x$, $y$, $z$ be at least two numbers (say, $x \le y$) less than $\frac{2}{3}$.
Replace these numbers by $x', y'$ in such a way that $x+y=x'+y'$, $[x', y'] \subset\left[0, \frac{2}{3}\right]$ and either $x'=0$ or $y'=\frac{2}{3}$.
By convexity $f(x') + f(y') + f(z) \ge f(x) + f(y) + f(z) > 0$.
If $x'=0$, then $y' \le \frac{2}{3}$, and $z \le 1$, thus $x'+y' + z < 2$. Hence $y' =\frac{2}{3}$.

Thus from the very beginning one can assume that $x \le\frac{2}{3}$ and $y, z \ge \frac{2}{3}$.
If $x = y = z = \frac{2}{3}$, then $f(x) + f(y) + f(z) = f\left(2 \cdot \frac{1}{3}\right) + 2f\left(1 - \frac{1}{3}\right) = 0$.
Repeating the same trick and using concavity of $f$ on $\left[\frac{2}{3}, 1\right]$ one can reduce the problem to the case $y=z$. But for any triple $x, y = z = 1 - \frac{x}{2}$ there holds $f(x) + f(y) + f(z) = 0$, which contradicts the assumption $f(x) + f(y) + f(z) > 0$.
\end{proof}

\subsection{Description of projections of measure classes on the triangle}

We will consider special classes of measures on $\Delta$ and describe their projections onto the axes.

First, consider the Lebesgue measure on $\Delta$. It can be normalized in such a way that the measure of the whole triangle is equal to $\frac{1}{2}$. We denote the normalized measure by $\lambda_\Delta$.
Projecting it to any hyperplane $\{x=0\}$, $\{y=0\}$, $\{z=0\}$, we get a triangle with the usual Lebesgue measure. In what follows we shall consider the densities with respect to this normalized measure.

\begin{definition}
Let $\mu$ be a measure on $\Delta$ absolutely continuous with respect to $\lambda_\Delta$. For any point $(x, y, z) \in \Delta$ define $M(x, y, z) = \min(1 - x, 1 - y, 1 - z)$. We call a measure $\mu$ \textbf{\textit{layered}} if for any $t$ the density of $\mu$ is constant on a set $M(x, y, z) = t$, that is density depends only on $M(x, y, z)$. 
\end{definition}

It is easy to see that $M$ is proportional to the distance from the point to the nearest side of $\Delta$. Therefore, points with constant $M$ form a triangle homothetic to the original one, with the same center. It is also easy to see that due to the symmetry of the layered measure, its projection on all three axes will be the same. Also note that $M$ takes values only in $\left[0, \frac{1}{3}\right]$.

\begin{definition}
We say that a function $p: \left[0, \frac{1}{3}\right] \to \mathbb{R}$ \textit{\textbf{generates}} a layered measure $\mu$ if $\frac{d\mu}{d\lambda_\Delta}(x, y, z) = p\left(M(x, y, z)\right)$.
\end{definition}

Let us find the projections of a layered measure $\mu$ generated by $p$ to the coordinate axes.

\begin{proposition}
 \label{prop:layered_projection}
Let $\mu$ be a layered measure generated by a function $p$. Let $p_*: [0, 1] \to \mathbb{R}_+$ be the density of the projection of this measure onto an axis. Then
\begin{align*}
p_*(x) = 
\begin{cases}
2\int_0^\frac{x}{2}p(t)dt,&\text{ if } x \le \frac{2}{3},\\
(3x - 2)p(1 - x) + 2\int_0^{1 - x}p(t)dt,&\text{ if } x \ge \frac{2}{3}
\end{cases}
\end{align*}
\end{proposition}
\begin{proof}
Denote the projection of $\mu$ onto the hyperplane $xy$ by $\mu_{xy}$. It is concentrated on the triangle $T$ with vertices $(0, 1)$, $(1, 0)$ and $(1, 1)$. Its density with respect to the Lebesgue measure on the plane at the point $(x, y)$ lying inside $T$ is $$p( M(x, y, 2 - x - y)) = p(\min(1 - x, 1 - y, x + y - 1)).$$

Define $\mu_x$ as the projection of $\mu$ onto $x$, or, what is the same, the projection of the measure $\mu_{xy}$ onto $x$. Then the measure of $[0, x_0]$ on the one hand is $\int_{0}^{x_0} p_*(x) dx$, and on the other hand is equal to the measure of the part of the triangle $T$ where the $x$ coordinate belongs to $[0, x_0]$. Thus, we have established the equality $\int_{0}^{x_0} p_*(x) dx = \int_{0}^{x_0} \int_{1-x}^1 p(M(x, y, 2 - x - y)) dx dy$. Differentiating both sides of this equality with respect to $x_0$, we obtain $p_*(x) = \int_{1 - x}^1 p(\min(1 - x, 1 - y, x + y - 1)) dy$. 

Assume $x \le \frac{2}{3}$. Then:
\begin{align*}
\min(1 - x, 1 - y, x + y - 1) = \begin{cases}
x + y - 1, &\text{ for } y \in \left[1 - x, 1 - \frac{x}{2}\right],\\
1 - y, &\text{ for } y \in \left[1 - \frac{x}{2}, 1\right].
\end{cases}
\end{align*}
From here we get: 
\begin{align*}
p_*(x) &= \int_{1 - x}^1 p(\min(1 - x, 1 - y, x + y - 1)) dy \\ 
&= \int_{1 - {x \over 2}}^1 p(1 - y) dy + \int_{1 - x}^{1 - {x \over 2}} p(x + y - 1) dy = 2\int_0^{x \over 2}p(t)~dt.
\end{align*}

Analogously for $x \ge \frac{2}{3}$:
\begin{align*}
\min(1 - x, 1 - y, x + y - 1) = \begin{cases}
x + y - 1, &\text{ for } y \in \left[1 - x, 2 - 2x\right],\\
1 - x, &\text{ for } y \in \left[2 - 2x, x\right],\\
1 - y, &\text{ for } y \in \left[x, 1\right].	
\end{cases}
\end{align*}
After this we calculate $p_*(x)$:
\begin{align*}
p_*(x) &= \int_{1 - x}^1 p(\min(1 - x, 1 - y, x + y - 1)) dy \\ 
&= \int_{1 - x}^{2 - 2x}p(x + y - 1)dy + \int_{2 - 2x}^{x}p(1 - x)dy + \int_{x}^{1}p(1 - y)dy \\ 
&= 2\int_0^{1 - x}p(t)dt + (3x - 2)p(1-x).
\end{align*} 
\end{proof}

Next we define \textit{median measure}.

\begin{definition}
\textit{The median subset} of $\Delta$ is the set $$\{(x, y, z) \in \Delta \mid x = y \ge z\} \cup \{(x, y, z) \in \Delta \mid y = z \ge x\} \cup \{(x, y, z) \in \Delta \mid x = z \ge y\}.$$ From a geometric point of view, this is a union of three segments in $\Delta$ from vertices to the center of the triangle $\Delta$.
\end{definition}
Projections of any segment from the median set are $\left[0, \frac{2}{3}\right]$ and $\left[\frac{2}{3}, 1\right]$.
On these segments one can define a measure proportional to the Lebesgue measure such that the measure of each segment is $\frac{2}{3}$. In what follows, we shall consider all the densities on the median set with respect to this measure.

\begin{definition}
\textit{Median measure} $\mu$, generated by a density function $q: [0, {2 \over 3}] \to \mathbb{R}_+$, is a measure with density on the median set that its density on each of the segments is equal to $q(t)$ at the points $(t, t, 2-2t)$, $(t, 2-2t, t)$, $(2-2t, t, t)$ with respect to the 
reference measure described above.
\end{definition}

It is easy to verify the following assertion:

\begin{proposition} \label{prop:median_projection}
Let $\mu$ be a median measure generated by $q$. Let $q_*(x)$ be the density of the projection of this measure onto an arbitrary axis. Then

\begin{align*}
q_*(x) = \begin{cases}
q(x), &\text{ for } x < {2 \over 3},\\
4q\left(2 - 2x\right), &\text{ for } x > {2 \over 3}.
\end{cases}
\end{align*}
\end{proposition}

This implies, in particular, the following identity

\begin{equation}\label{eq:median_measure_projection}
4q_*(2x) = q_*(1 - x), \ x < \frac{1}{3}.
\end{equation}

The converse is also true: if nonnegative $q_*$ satisfies \cref{eq:median_measure_projection}, then there is a median measure which projection onto arbitrary axis coincides with $q_*$.

\subsection{Combining measures}

Let $\pi$ be a measure on the segment $[0, 1]$ with density $f$. We are concerned with $f(x)=\alpha^x$, but we will only use the fact that $f(x)$ is continuously differentiable, increasing, convex and satisfies $\int_0^1 \left(t - \frac{2}{3}\right)f(t) dt = 0$. The last means that measure with density $f$ satisfies \cref{main_traingle_equality}. 

We want to find a measure $\mu$ that is the sum of the layered measure $\mu_p$ generated by a function $p$ and the median measure $\mu_q$ generated by a function $q$, whose projection on each of the axes coincides with $\pi$.

We subtract $\mu_p$ from $\mu$ and look at the projection of $\mu - \mu_p$ on the axes with the density $q_*(x)$. By \cref{prop:layered_projection}, the projection is equal to
\begin{align*}q_*(x)=
\begin{cases}
f(x) - 2\int_0^\frac{x}{2}p(t)dt, &\text{ for } {x \le \frac{2}{3}},\\
f(x) - (3x - 2)p(1 - x) - 2\int_0^{1 - x}p(t)dt, &\text{ for } x \ge \frac{2}{3}.
\end{cases}
\end{align*}

In order for $q_*(x)$ to be a density of the projection of a median measure, it suffices that $q_*(x) \ge 0$ and $4q_*(2x) = q_*(1 - x)$ for $x \le \frac{1}{3}$. Using the identities on $q_*(x)$ given above, we obtain the equivalent equation:

\begin{equation} \label{eq:diffur}
4\left(f(2x) - 2 \int_0^x p(t)dt\right) = f(1 - x) - (1 - 3x)p(x) - 2\int_0^xp(t)dt
\end{equation}

Assuming $P(x) = \int_0^x p(t)dt$, we get the following equation
$$
4\left(f(2x) - 2P(x)\right) = f(1 - x) - (1 - 3x)P'(x) - 2P(x) 
$$

This is a differential equation of the first degree, its solutions have the form

$$
P(x) = \frac{c_1 + \int_0^x(1 - 3t)(f(1-t) - 4f(2t))dt}{(1 - 3x)^2}.
$$

Using $P(0) = 0$ we get $c_1 = 0$, 

$$
P(x) = \frac{1}{(1 - 3x)^2} \int_0^x(1 - 3t)(f(1-t) - 4f(2t))dt.
$$

Now suppose that $f$ is continuously differentiable. We find $p(x)$ using integration by parts:
\begin{align*}
p(x) &= P'(x) \\ 
&= \frac{(f(1 - x)-4f(2x))(1 - 3x)^2 + 6\int_0^x(1 - 3t)(f(1-t) - 4f(2t))dt}{(1 - 3x)^3} \\ 
&= \frac{1}{(1 - 3x)^3}\left(f(1) - 4f(0) - \int_0^x(1 - 3t)^2(f'(1-t) + 8f'(2t))dt\right).
\end{align*}

To prove that $p(x)$ and corresponding $q(x)$ generate a nonnegative density, we need to check that $p(x) \ge 0$ for $x \in \left[0, \frac{1}{3}\right]$, $p\left(\frac{1}{3}\right)$ is well-defined and $f(2x)-2P(x) = q_*(2x) = q(2x) \ge 0$ for $x \in \left[0, \frac{1}{3}\right]$, where $q(x)$ generates the median measure.

\begin{lemma}\label{lem:1_3_well_defined}
Suppose that $f : [0, 1] \to \mathbb{R}$ is a continuously differentiable monotonically increasing function and $\int_0^1 \left(t - \frac{2}{3}\right)f(t) dt = 0$. Then the function
$$
I(x) = f(1) - 4f(0) - \int_0^x(1 - 3t)^2(f'(1-t) + 8f'(2t))dt
$$
is nonnegative on $\left[0, {1 \over 3}\right]$ and $I\left(1 \over 3\right) = 0$.
\end{lemma}
\begin{proof}
Since $f$ is increasing, $f' \ge 0$ and the integrand $(1 - 3t)^2(f'(1-t) + 8f'(2t))$ is nonnegative. So the integral increases and $I(x)$ monotonically decreases to $I\left(\frac{1}{3}\right)$. Integrating by parts we get

\begin{align*}
I\left(1 \over 3\right) &= f(1) - 4f(0) - \int_0^\frac{1}{3}(1 - 3t)^2(f'(1-t) + 8f'(2t)) dt \\
&= \int_0^\frac{1}{3} (4f(2t) - f(1-t)) d(1-3t)^2\\
&= 6\int_0^\frac{1}{3} (1 - 3t)(f(1 - t) - 4f(2t)) dt\\
&= 18\int_0^1 \left(t - \frac{2}{3}\right)f(t) dt \\
&= 0.
\end{align*}
\end{proof}

Using this lemma one can check that $p(x)$ is nonnegative and well-defined.
\begin{proposition}\label{prop:pcorrect}
Suppose that $f(x)$ satisfies the conditions of \cref{lem:1_3_well_defined}. Then the function
$$
p(x) = \frac{1}{(1 - 3x)^3}\left(f(1) - 4f(0) - \int_0^x(1 - 3t)^2(f'(1-t) + 8f'(2t))dt\right)
$$
is nonnegative and $\lim_{x \to {1 \over 3}}p(x) = f'\left(2 \over 3\right)$.
\end{proposition}
\begin{proof}
Using the function $I(x)$ from \cref{lem:1_3_well_defined} we can rewrite the function $p(x)$ as follows:
$$
p(x) = \frac{I(x)}{(1 - 3x)^3}.
$$
$I(x)$ is nonnegative, so is $p(x)$. Let us check that $p\left(1 \over 3\right)$ is well-defined.

Since $I\left(1 \over 3\right) = 0$ one can apply the L'Hospital rule to $p(x)$:
\begin{align*}
    \lim_{x \to \frac{1}{3}}p(x) &= \lim_{x \to \frac{1}{3}}\frac{I(x)}{(1 - 3x)^3} =  \lim_{x \to \frac{1}{3}}-\frac{I'(x)}{9(1 - 3x)^2}\\
    &=\lim_{x \to \frac{1}{3}}\frac{(1 - 3x)^2(f'(1-x) + 8f'(2x))}{9(1 - 3x)^2}\\
    &=\frac{1}{9}\lim_{x \to \frac{1}{3}}(f'(1-x) + 8f'(2x)) = f'\left({2 \over 3}\right).
\end{align*}
\end{proof}

Now we will check that the function $q$ is nonnegative as well, so it generates the measure with nonnegative density.
\begin{proposition}\label{prop:qcorrect}
Suppose that $f(x)$ satisfies the conditions of \cref{lem:1_3_well_defined} and $f(x)$ is convex on $[0, 1]$. Then the function
$$
q(2x) = f(2x) - 2P(x) = f(2x) - \frac{2}{(1 - 3x)^2} \int_0^x(1 - 3t)(f(1-t) - 4f(2t))dt
$$
is nonnegative.
\end{proposition}
\begin{proof}
Write the function $q$ in the following form:
\begin{align*}
    q(2x) &= f(2x) - \frac{2}{(1 - 3x)^2} \int_0^x(1 - 3t)(f(1-t) - 4f(2t))dt \\
    &= f(2x) + \frac{1}{3(1 - 3x)^2}\int_0^x(f(1-t) - 4f(2t))d(1 - 3t)^2\\
    &= f(2x) + \frac{(1 - 3t)^2(f(1 - t) - 4f(2t))|_0^x - \int_0^x(1 - 3t)^2(f'(1 - t) + 8f'(2t))dt}{3(1 - 3x)^2}\\
    &= \frac{(1 - 3x)^2(f(1 - x) - f(2x)) - I(x)}{3(1 - 3x)^2}.
\end{align*}

To check that $q \ge 0$ it suffices to check that the numerator $n(x) = (1 - 3x)^2(f(1 - x) - f(2x)) - I(x)$ is nonnegative. From \cref{lem:1_3_well_defined} $n\left(\frac{1}{3}\right) = 0$. So we check that $n$ is decreasing.

\begin{align*}
n'(x) &= 6(1 - 3x)(f'(2x)(1 - 3x) - (f(1 - x) - f(2x))),\\
n'(x) &\le 0 \Leftrightarrow f'(2x) \le \frac{f(1 - x) - f(2x)}{1 - 3x}.
\end{align*}

The last equality holds since $f$ is convex.
\end{proof}

Summarizing the last two propositions we obtain the following theorem:

\begin{theorem}
For any continuously differentiable, increasing and convex function $f: [0, 1]$ satisfying $\int_0^1 \left(t - \frac{2}{3}\right)f(t) dt = 0$, there exists a measure on $\Delta$ with projections onto the axes have densities $f(x)$.
\end{theorem}

All the assumptions can be applied to $f(x) = \alpha^x$, where $\alpha$ is a solution of \cref{alpha_equation}.

Also we find $p(x)$ and $q(x)$ explicitly.

\begin{align*}
p(x) &= -\frac{1}{(1 - 3x)^3}\int_0^x(1 - 3t)^2(f'(1-t) + 8f'(2t))dt \\
&=\frac{\alpha^{1-x}-4 \alpha^{2 x}}{1-3 x}-6\frac{2 \alpha^{2 x}+\alpha^{1-x}}{(1-3x)^2 \ln\alpha}-6\frac{3 \alpha^{2 x}+3\alpha-\alpha
\ln\alpha-3-2\ln\alpha-3 \alpha^{1-x}}{(1-3 x)^3 \ln ^2\alpha}\\
&= \frac{\alpha^{1-x}-4 \alpha^{2 x}}{1-3 x}-6\frac{2 \alpha^{2 x}+\alpha^{1-x}}{(1-3x)^2 \ln\alpha}-18\frac{\alpha^{2 x}-\alpha^{1-x}}{(1-3 x)^3 \ln ^2\alpha},
\end{align*}
\begin{align*}
q(2x) &= f(2x) - 2P(x) = f(2x) - {2 \over (1 - 3x)^2}\int_0^x(1 - 3t)(f(1-t) - 4f(2t))dt \\
&=\alpha^{2 x}+2\frac{2 \alpha^{2 x}+\alpha^{1-x}}{(1-3 x) \ln
\alpha}+2\frac{3 \alpha^{2 x}+3\alpha-\alpha \ln\alpha-3
-2 \ln\alpha-3\alpha^{1-x}}{(1-3 x)^2 \ln ^2\alpha} \\
&= \alpha^{2 x}+2\frac{2 \alpha^{2 x}+\alpha^{1-x}}{(1-3 x) \ln
\alpha}+6\frac{\alpha^{2 x}-\alpha^{1-x}}{(1-3 x)^2 \ln ^2\alpha}.
\end{align*}

The last identities follow from \cref{alpha_equation}.

Now we are ready to present the main theorem of this section:

\begin{theorem}
\label{thm:osexists}
There exists a $(3, 1)$-stochastic measure concentrated on the set $M$. 
\end{theorem}
\begin{proof} 
Let us collect all the details of the proof together and describe our measure explicitly. Set $M$ contains segments connecting points $(0, 1, 1)$ and $(l, r, r)$, $(1, 0, 1)$ and $(r, l, r)$, $(1, 1, 0)$ and $(r, r, l)$. This segments have length $L = \sqrt{l^2 + 2(1 - r)^2}$. 
Define measure $\mu_{lin}$ as a sum of Lebesgue measures on this segments divided by$\frac{l}{L}$. 

The projections of two segments coincide with $[r, 1]$, the densities are equal to $\frac{L}{1 - r} \cdot \frac{l}{L} = \frac{1}{2}$. Their sum is the Lebesgue measure on $[r, 1]$.
The projection of the third interval is a measure on $[0, l]$, its density equals $\frac{L}{l} \cdot \frac{l}{L} = 1$.

The mapping $$u = {\ln x - \ln l \over \ln r - \ln l}, \ v = {\ln y - \ln l \over \ln r - \ln l}, \ w = {\ln z - \ln l \over \ln r - \ln l}$$ 
transforms the two-dimensional part of $M$ into triangle $\Delta$. 
We equip $\Delta$ with the layered measure $\mu_p$ generated by $$p(x) = \frac{\alpha^{1-x}-4 \alpha^{2 x}}{1-3 x}-6\frac{2 \alpha^{2 x}+\alpha^{1-x}}{(1-3x)^2 \ln\alpha}-18\frac{\alpha^{2 x}-\alpha^{1-x}}{(1-3 x)^3 \ln ^2\alpha},$$ and the median measure $\mu_q$ generated by $$q(2x) = \alpha^{2 x}+2\frac{2 \alpha^{2 x}+\alpha^{1-x}}{(1-3 x) \ln\alpha}+6\frac{\alpha^{2 x}-\alpha^{1-x}}{(1-3 x)^2 \ln^2\alpha}.$$

Then by \cref{prop:layered_projection} the projection of $\mu_p$ coincides with
\begin{align*}
\begin{cases}
2\int_0^\frac{x}{2}p(t)dt, &\text{ for } {x \le \frac{2}{3}},\\
(3x - 2)p(1 - x) + 2\int_0^{1 - x}p(t)dt, &\text{ for } x \ge \frac{2}{3}.
\end{cases}
\end{align*}

Since $p$ is a solution of \cref{eq:diffur} for $f(x) = \alpha^x$ we can conclude that for
\begin{align*}q_*(x) =
\begin{cases}
f(x) - 2\int_0^\frac{x}{2}p(t)dt, &\text{ for } {x \le \frac{2}{3}},\\
f(x) - (3x - 2)p(1 - x) - 2\int_0^{1 - x}p(t)dt, &\text{ for } x \ge \frac{2}{3}.
\end{cases}
\end{align*}
there holds $4q_*(2x) = q_*(1 - x)$. Thus by \cref{prop:median_projection} $q_*(x)$ is the projection of $\mu_q$ generated by $q(2x) = f(2x) - 2\int_0^\frac{x}p(t)dt = \alpha^{2 x}+2\frac{2 \alpha^{2 x}+\alpha^{1-x}}{(1-3 x) \ln
\alpha}+6\frac{\alpha^{2 x}-\alpha^{1-x}}{(1-3 x)^2 \ln ^2\alpha}$.

By \cref{prop:pcorrect} and \cref{prop:qcorrect} this construction is well-defined. Projections of $\mu_p + \mu_q$ on axes coincide with $\alpha^x$ in coordinates $u, v, w$ and with the uniform measure on $[l, r]$ in initial coordinates.

Thus the projections of $\mu = \mu_p + \mu_q + \mu_{lin}$ coincide with Lebesgue measure on $[0, 1]$.
\end{proof}
\section{The dual solution construction}\label{section:dual}
To prove that the measure $\mu$ from \cref{thm:osexists} is the primal solution it is enough to find a triple of functions $f, g, h: [0, 1] \to \mathbb{R}$ such that $f(x) + g(y) + h(z) \le c(x, y, z)$ and equality holds on the set $M$ by \cref{lem:slackness_conditions}. In this case the triple $(f, g, h)$ will be a dual solution of the related problem. In this section we will construct the dual solution for a wide class of cost functions.

We will construct the dual solution for $c(x, y, z) = \widehat{C}(\ln x + \ln y + \ln z)$ where $\widehat{C}$ is a bounded continuously differentiable strictly convex function on $(-\infty, 0]$. Our function $c(x, y, z) = xyz$ is a partial case for $\widehat{C}(t) = \exp(t)$. At the same time we will use the more convenient equivalent description. Namely, $c(x, y, z) = C(xyz)$ for some continuously differentiable function $C:[0, 1] \to \mathbb{R}$ and the function $tC'(t)$ strictly increases on the segment $[0, 1]$.

\subsection{Another description of the  support of primal solutions}
\begin{definition}\label{def:lambda_definition}
Set $c = lr^2$. Define a function $\lambda: [0, 1] \to \mathbb{R}$ as follows
\begin{align*}
\lambda(x)=\begin{cases}
(1-2x)^{2} & \text{if } x \in [0, l)\\
\frac{c}{x} & \text{if } x \in [l, r),\\
\frac{1}{2}x(1-x) & \text{if } x \in [r, 1].
\end{cases}
\end{align*}
\end{definition}
\begin{lemma}
The function $\lambda$ defined above is continuous and strictly decreases.
\end{lemma}
\begin{proof}
It suffices to check the continuity at points $l$ and $r$. For this it suffices to check that $(1 - 2l)^{2} = \frac{c}{l}$ and $\frac{c}{r} = \frac{1}{2}r(1-r)$.
All these equalities are trivial. 

Let us check that the derivative of $\lambda$ is negative everywhere except of the points $l$ and $r$: in these points $\lambda$ has no derivatives.

If $x \in (0, l)$ then $\lambda'(x) = 2(2x-1) < 0$, since $x < l < \frac{1}{2}$. If $x \in (l, r)$ then
$\lambda'(x)=-\frac{c}{x^{2}} < 0$ since $c > 0$. If $x \in (r, 1)$ then $\lambda'(x) = 1 - \frac{1}{2}x < 0$ since $x > r > \frac{1}{2}$. 

It follows from this that $\lambda$ strictly decreases.
\end{proof}

\begin{proposition}\label{prop:description_M}
Suppose that $M$ is the (hypothetical) primal solution support as in the previous sections. Then a point $(x, y, z)$ is contained in $M$ if and only if the following equalities hold
$\lambda(x)=yz$, $\lambda(y)=xz$, $\lambda(z)=xy$.
\end{proposition}
\begin{proof}
$\Leftarrow$ Suppose that $\kappa(x)=x\lambda(x)$. If $\lambda(x)=yz$,
$\lambda(y)=xz$ and $\lambda(z)=xy$ then $\kappa(x)=\kappa(y)=\kappa(z)=xyz$. 

The function $\kappa(x)$ is continuous and has a continuous derivative on intervals $(0, l)$, $(l, r)$ and $(r, 1)$. If $x \in (0, l)$ then $\kappa'(x) = (1-2x)^{2} - 2x(1 - 2x) = (1 - 2x)(1 - 4x) > 0$
since $x < l < \frac{1}{4}$. On the segment $[l, r]$ $\kappa$ is constant: $\kappa(x) = lr^2 = c$.
If $x \in (r, 1)$ then $\kappa'(x) = x(1-x) - \frac{1}{2}x^2=x\left(1 - \frac{3}{2}x\right) < 0$ since $x > r > \frac{2}{3}$. So $\kappa(x)$ strictly increases on the segment $[0, l]$, is constant on $[l, r]$, and strictly decreases on $[r, 1]$.

Note in addition that $\kappa(0) = \kappa(1) = 0$. Thus, the equation $\kappa(x) = c_0$ for $0 \le x \le 1$
\begin{enumerate}
    \item has no root if $c_0 < 0$ or $c_0 > c$;
    \item has exactly two roots if $0 \le c_0 < c$: one of them lies on the interval $[0, l)$ and another one lies on the interval $(r, 1]$;
    \item holds on whole segment $[l, r]$ if $c_0 = c$.
\end{enumerate}

If $\lambda(x) = yz$, $\lambda(y) = xz$, and $\lambda(z) = xy$ then $\kappa(x)=\kappa(y)=\kappa(z)=xyz$ and one of the following cases occurs:
\begin{enumerate}
\item $x, y, z \in [l, r]$. In this case $c = \kappa(x) = \kappa(y) = \kappa(z) = xyz$ so $(x, y, z) \in M$.
\item $x = y = z \in[0, l)$. Then $\lambda(x)=x^2$. On the other hand if $x\in[0, l)$ then $\lambda(x)=(1-2x)^2$. The equation $(1-2x)^{2}=x^{2}$
has two solutions $x = 1$ and $x = \frac{1}{3}$. But these values are not feasible because
$x \in [0, l)$ and $l < \frac{1}{6}$. So, this case is not possible.
\item $x=y=z\in(r, 1]$. Similarly in this case $\lambda(x)=x^2$. On the other hand if $x \in (r, 1]$ then $\lambda(x)=\frac{1}{2}x(1-x)$.
Equation $\frac{1}{2}x(1-x)=x^2$ has two solutions $x = 0$ and $x = \frac{1}{3}$, but they do not belong 
$(r, 1]$ for any $r > \frac{1}{2}$. So, this case is not possible.
\item $x = y \in [0, l)$, $z\in(r, 1]$ and similar cases obtained by permutations of coordinates. One has $x(1-2x)^2 = \kappa(x) = \kappa(z) = \frac{1}{2}z^2(1-z)$.
The function $\kappa(z)$ strictly decreases on the interval $(r, 1]$, hence  for a fixed $x$ there exists at most one $z$ satisfying this equality. But $z=1-2x \in (r, 1]$ and $\kappa(z) = \frac{1}{2}z^2(1-z) = \frac{1}{2}(1-2x)^2 \cdot 2x = \kappa(x)$.
This means that $z=1-2x$. In this case $x(1-2x) = \frac{1}{2} z(1-z) = \lambda(z)=xy = x^2$.
Hence $x=0$ or $x=\frac{1}{3}$. But for $x=0$ one has $1 = \lambda(x) = yz = xz = 0$. The value $x=\frac{1}{3}$
 is not suitable because $x \in [0, l)$ and $l < {1 \over 6}$. So, this case is not possible.
\item $x \in [0, l)$, $y = z \in (r, 1]$ and  similar cases obtained by permutations of coordinates. Arguing as above, we get $\kappa(x) = \kappa(z)$,
$x \in [0, l)$, $z \in (r, 1]$ so $y = z = 1-2x$. The points $(x, 1-2x, 1-2x)$ are contained in $M$ for any $x \in [0, l)$.
\end{enumerate}
So, the only possible cases are cases 1 and 5. In these cases $(x, y, z) \in M$.

$\Rightarrow$ The set $M$ consists of four parts: $M = M_x \cup M_y \cup M_z \cup M_2$. If $(x, y, z) \in M_x$, then $y = z = 1 - 2x$. Hence $\lambda(x) = (1 - 2x)^2$ and $yz = (1 - 2x)^2$. $\lambda(y) = \lambda(1 - 2x) = {1 \over 2}(1 - 2x)\cdot 2x = x(1 - 2x) = xz$ since $r \le 1 - 2x \le 1$. Similarly $\lambda(z) = xy$.

Hence if $(x, y, z) \in M_x$, then $\lambda(x) = yz$, $\lambda(y) = xz$ and $\lambda(z) = xy$. By symmetry, these conditions hold for any $(x, y, z) \in M_y$ and for any $(x, y, z) \in M_z$.

If $(x, y, z) \in M_2$, then $l \le x, y, z \le r$ and $xyz = c$. This means that $\lambda(x) = {c \over x} = yz$, $\lambda(y) = {c \over y} = xz$, $\lambda(z) = {c \over z} = xy$.
\end{proof}

\subsection{The construction of the dual solution}

If $M$ is indeed a support of the primal solution and $f, g, h$ is a dual solution, then by complementary slackness $f(x)+g(y)+h(z)$ is equal to $c(x, y, z)$ on almost all points of $M$. This will help us to guess the form of $f, g, h$.

\begin{lemma}\label{thm:weak_uniqueness}
Assume that $c(x, y, z) = C(xyz)$ for some continuously differentiable function $C:[0, 1] \to \mathbb{R}$
and  the  triple of functions 
$$f, g, h: [0, 1] \to \mathbb{R}
$$ satisfies inequality $f(x) + g(y) + h(z) \le c(x, y, z)$ and $f(x) + g(y) + h(z) = c(x, y, z)$ for all $(x, y, z) \in M$.
Then the functions $f, g, h$ are continuously differentiable and $f'(x)=\lambda(x)C'(x\lambda(x))$,
$g'(y)=\lambda(y)C'(y\lambda(y))$, $h'(z)=\lambda(z)C'(z\lambda(z))$.
\end{lemma}
\begin{proof}
For any $x_0$ there exist $y_0$ and $z_0$ such that $(x_0, y_0, z_0) \in M$. This means that $f(x_0) + g(y_0) + h(z_0) = c(x_0, y_0, z_0) = C(x_0\lambda(x_0))$.
In addition, for any $x$ one has  
$$f(x) + g(y_0) + h(z_0) \le c(x, y_0, z_0) = C(x\lambda(x_0)).
$$
Hence for any $x_0, x \in [0, 1]$ one has
$f(x) - f(x_0) \le  C(x\lambda(x_0))-C(x_0\lambda(x_0))$. 

Passing to the limit $x \to x_0$  one gets
\begin{align*}
C(x\lambda(x_0))-C(x_0\lambda(x_0)) = (x - x_0) \cdot \lambda(x_0)C'(x_0\lambda(x_0)) + o(|x - x_0|).
\end{align*}

Interchanging $x_0$ and $x$ one gets $f(x_0) - f(x) \le C(x_0\lambda(x)) - C(x\lambda(x))$. By the mean value theorem, $C(x_0\lambda(x)) - C(x\lambda(x)) = (x_0 - x)\lambda(x)C'(\xi(x))$, where $\xi(x) \in [x_0\lambda(x), x\lambda(x)]$. If $x \to x_0$, then $\xi(x) \to x_0\lambda(x_0)$ and 
\begin{align*} 
C(x_0\lambda(x))-C(x\lambda(x)) & = (x_0 - x) \lambda(x)C'(x_0\lambda(x_0)) + o(|(x_0 - x) \lambda(x)|) 
\\& = (x_0 - x) \lambda(x)C'(x_0\lambda(x_0)) + o(|x - x_0|) 
\\& = (x_0 - x) \lambda(x_0)C'(x_0\lambda(x_0)) + o(|x - x_0|).
\end{align*}

This means that 
\begin{multline*}
    \lambda(x_0)C'(x_0\lambda(x_0)) \cdot (x - x_0) + o(|x - x_0|) \\
    \le f(x) - f(x_0) \\
    \le \lambda(x_0)C'(x_0\lambda(x_0)) \cdot (x - x_0) + o(|x - x_0|).
\end{multline*}
Hence $f(x)$ has a derivative at the point $x = x_0$ and it is equal to $\lambda(x_0)C'(x_0\lambda(x_0))$. This function is continuous since $\lambda$ and $C'$ are continuous.

One can check in the same way  the statements of the theorem for the functions $g$ and $h$.
\end{proof}

\begin{theorem}
Suppose that $c(x, y, z) = C(xyz)$ for some continuously differentiable function $C:[0, 1] \to \mathbb{R}$ and the function $U(t) = tC'(t)$ strictly increases on the segment $[0, 1]$. 
Suppose that $\hat{f}(s) = \int_{0}^{s} \lambda(t)C'(t\lambda(t))~dt$.
Then the arg max of the function $\hat{f}(x) + \hat{f}(y) + \hat{f}(z) - c(x, y, z)$ contains the set $M$.
\end{theorem}
\begin{proof}
Assume that $T(x, y, z) = \hat{f}(x) + \hat{f}(y) + \hat{f}(z) - c(x, y, z) = \hat{f}(x) + \hat{f}(y) + \hat{f}(z) - C(xyz)$.
If $(x, y, z) \in M$ then 
\begin{equation*}
\nabla T(x, y, z) =
\begin{pmatrix}
\lambda(x)C'(x\lambda(x)) - yzC'(xyz)\\
\lambda(y)C'(y\lambda(y)) - xzC'(xyz)\\ \lambda(z)C'(z\lambda(z)) - xyC'(xyz)
\end{pmatrix}= \vec{0}.
\end{equation*}
Hence, all values of $T$ on the set $M$ are the same since $M$ is path-connected.

The function $T$ is continuous on the compact set $[0, 1]^{3}$, so the function $T$ reaches its maximum at some point $(x_0, y_0, z_0)$. Then either
$x_0$ lies on the boundary of the segment $[0, 1]$ or $\frac{\partial T}{\partial x}(x_0, y_0, z_0) = 0$. 

For any $x > 0$ the following equality holds
\[
\frac{\partial T}{\partial x}(x, y_0, z_0) = \lambda(x)C'(x\lambda(x))-y_0z_0C'(xy_0z_0)=\frac{U(x\lambda(x)) - U(xy_0z_0)}{x}.
\]

Assume that $x_0 = 0$. By the mean value theorem
for any $x > 0$ there exists $0 < \xi(x) < x$ such that
\begin{multline*}
T(x, y_0, z_0) - T(x_0, y_0, z_0) = x\frac{\partial T}{\partial x}(\xi(x), y_0, z_0)\\
=\frac{x}{\xi(x)}\left(U[\xi(x)\lambda(\xi(x))]-U[\xi(x)y_0z_0]\right).
\end{multline*}
One has
$T(x, y_0, z_0) \le T(x_0, y_0, z_0)$ since $(x_0, y_0, z_0)$ is a maximum  point of $T$. Hence, $U[\xi(x)\lambda(\xi(x))] \le U[\xi(x)y_0z_0]$ and $\xi(x)\lambda(\xi(x)) \le \xi(x)y_0z_0$ since $U$ strictly increases. This means that $\lambda(\xi(x)) \le y_0z_0$ for all $x > 0$. If $x \to 0$ then $\lambda(\xi(x)) \to \lambda(0) = 1$. Thus $y_0z_0 \ge 1 \Rightarrow \lambda(x_0) = 1 = y_0z_0$.

Suppose that $x_0 = 1$. In this case $\frac{\partial T}{\partial x}(x_{0},y_{0},z_{0})$
must be nonnegative. But $\frac{\partial T}{\partial x}(x_0, y_0, z_0) = \frac{U(x_0\lambda(x_0)) - U(x_0y_0z_0)}{x_0} = U(0) - U(y_0z_0)$.
The function $U(t)$ strictly increases, hence $y_0z_0 = 0$. This implies $0 = \lambda(x_{0}) = y_0z_0$.

Otherwise one has $\frac{\partial T}{\partial x}(x_0, y_0, z_0) = \frac{1}{x_0}(U(x_0\lambda(x_0))-U(x_0y_0z_0)) = 0$. The function $U(t)$ strictly increases. Hence $x_0\lambda(x_0)=x_0y_0z_0$ and  $\lambda(x_0)=y_0z_0$.

Consequently, if the function $T$ has  maximum at the point $(x_0, y_0, z_0)$, one gets $\lambda(x_0) = y_0z_0$. Similarly, one can prove that $\lambda(y_0) = x_0z_0$
and $\lambda(z_0) = x_0y_0$. Hence by \cref{prop:description_M} $(x_{0},y_{0},z_{0}) \in M$. Since  $T$ is constant on  $M$, one has $M \subset \arg\max T$.
\end{proof}

Summarizing the results from the last two sections we get 

\begin{theorem}\label{thm:primal_dual}
Suppose that $c(x, y, z) = C(xyz)$ for some continuously differentiable function $C:[0, 1] \to \mathbb{R}$ and the function $tC'(t)$
strictly increases on the segment $[0, 1]$. Set: 
$$\hat{f}(s) = \int_{0}^{s}\lambda(t)C'(t\lambda(t))~dt.$$ Then for any constants $C_x$, $C_y$, $C_z$ such that $$C_x + C_y + C_z = C(0) - 2\int_{0}^{1}\lambda(t)C'(t\lambda(t))~dt$$ the following inequality holds 
$$(\hat{f}(x) + C_x) + (\hat{f}(y) + C_y) + (\hat{f}(z) + C_z) \le c(x, y, z)$$
with equality on $M$.

This means by \cref{lem:slackness_conditions} that the triple $(\hat{f} + C_x,~\hat{f} + C_y,~\hat{f} + C_z)$ is the dual solution for the cost function $c(x, y, z)$ and any probability measure $\mu$ such that $\Pr_X(\mu) = \Pr_Y(\mu) = \Pr_Z(\mu) = \lambda$ and $\supp(\mu) \subset M$ is the primal solution to the related problem.

Moreover such a measure $\mu$ exists by \cref{thm:osexists}.
\end{theorem}

We note that any primal solution is universal in a sense it is the same for the cost functions of type $C(xyz)$ where $tC'(t)$ is strictly increasing on $[0, 1]$. It is important for the proof that $M$ is path-connected. Numerical experiments for other marginals show that sometimes the support of a primal solution is not necessarily path-connected. For example for a measure $SF$ on $[0, 5]$ given by a density

\begin{align*}
\rho_{SF}(t) = \begin{cases}
\frac{1}{15} \text{, if } t \in [0, 1] \cup [2, 3] \cup [4, 5],\\
\frac{2}{5} \text{, if } t \in (1, 2) \cup (3, 4),
\end{cases}
\end{align*}
primal solution (more precisely the result of \cref{alg:primal_solution_approximation_faster}) for the cost function $c(x, y, z) = xyz$ is pictured on \cref{fig:SF}.

\begin{figure}[h]
\centering \label{fig:SF}
\includegraphics*[scale=0.3]{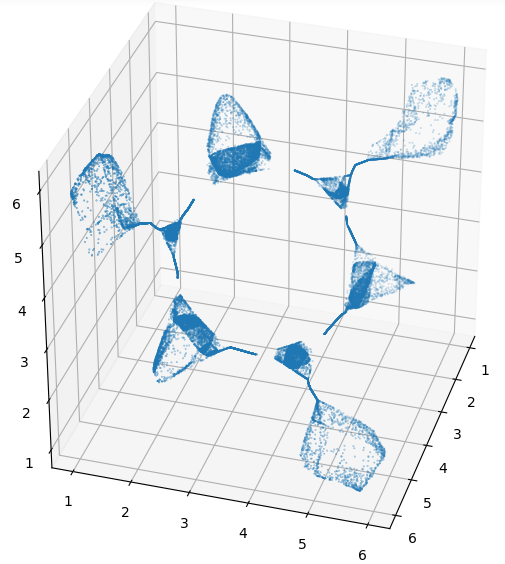} \\
\caption{Primal solution for marginals $SF$}
\end{figure}

\subsection{Construction for the cost function $c(x, y, z) = xyz$}\label{subsection:explicit_dual_xyz}

Suppose that $$c(x, y, z) = xyz = C(xyz),$$ 
where $C(t) = t$, $0 \le t \le 1$. The function $C(t)$ is continuously differentiable and $tC'(t) = t$ strictly increases. \cref{thm:primal_dual} implies  that any probability measure $\mu$ with projections $\Pr_X(\mu) = \Pr_Y(\mu) = \Pr_Z(\mu) = \lambda$ and $\supp(\mu) \subset M$ is the primal solution to the related problem; in particular the probability measure from \cref{thm:osexists} is the primal solution. Also, we can construct explicitly the dual solution in this case.

Consider the following functions:
\begin{align*}
&f_1(x) = c \ln l - {1 \over 3}(c\ln c - c) + {1 \over 6}((2x - 1)^3 - (2l - 1)^3),\\ 
&f_2(x) = c\ln x - {1 \over 3}(c\ln c - c),\\ 
&f_3(x) = c\ln r - {1 \over 3}(c\ln c - c) + {1 \over 4}(x^2 - r^2) - {1 \over 6}(x^3 - r^3).
\end{align*}

These functions satisfy the following identities:
$$f_1(l) = f_2(l),$$
$$f_2(r) = f_3(r),$$
$$f_1'(l) = f_2'(l),$$
$$f_2'(r) = f_3'(r).$$ 

The first and the second equality are easy to check directly. For the third and the fourth compute $f_1'(x) = (2x - 1)^2$, $f_2'(x) = \frac{c}{x}$, $f_3'(x) = \frac{1}{2}(x - x^2)$.
$f_1'(l) = (2l - 1)^2 = r^2 = \frac{c}{l} = f_2'(l)$, $f_2'(r) = lr = \frac{1}{2}r(1 - r) = f_3'(r)$.

Define:
\begin{align*}
f(x) = g(x) = h(x) = \begin{cases}
f_1(x) \text{, if } 0 \le x \le l,\\
f_2(x) \text{, if } l \le x \le r,\\
f_3(x) \text{, if } r \le x \le 1,
\end{cases}
\end{align*}

It follows immediately from the properties checked above of the functions $f_1, f_2, f_3$ that $f$ is continuous and continuously differentiable on $[0, 1]$ and $f'(x) = \lambda(x)$.

\begin{proposition}\label{prop:explicit_dual_xyz}
The triple of functions $(f, g, h)$ defined above is a dual solution of related problem for the cost function $c(x, y, z) = xyz$.
\end{proposition}

\begin{proof}
Since $f'(x) = \lambda(x)$ it follows that 
$$
f(x) = g(x) = h(x) = \int_0^x\lambda(x)~dx + C_f = \int_0^x\lambda(x)C'(x\lambda(x))~dx + C_f
$$
for some constant $C_f$. By \cref{thm:primal_dual} it is enough to check that $f(0) + f(1) + f(1) = c(0, 1, 1) = 0$.
\begin{align*}
    f(0) = f_1(0) &= c\ln l - {1 \over 3}(c\ln c - c) - {1 \over 6}(2l - 1)^3 - {1 \over 6},\\
    f(1) = f_3(1) &= c\ln r - {1 \over 3}(c\ln c - c) - {1 \over 4}r^2 + {1 \over 6}r^3 + {1 \over 12},\\
    f(0) + 2f(1) &= c\ln(lr^2) - (c\ln c - c) + 2 \cdot{1 \over 12} - {1 \over 6} - {1 \over 2}r^2
    + {1 \over 3}r^3 - {1 \over 6}(2l - 1)^3\\
    &= c - {1 \over 2}r^2 + {1 \over 2}r^3 = c - {1 - r \over 2}r^2 = c - lr^2 = 0.
\end{align*}
So the triple $(f, g, h)$ is the dual solution for the cost function $c(x, y, z)=xyz$.
\end{proof}

\section{Uniqueness of the dual solution}\label{section:uniqueness}
\begin{theorem}\label{thm:uniqueness}
Suppose that $c(x, y, z) = C(xyz)$ for some continuously differentiable function $C:[0, 1] \to \mathbb{R}$ and the function $tC'(t)$
strictly increases on the segment $[0, 1]$. Then the triple $(f, g, h)$ is a dual solution if and only if there exist constants $C_f$, $C_g$, $C_h$ such that 
 $$
C_f + C_g + C_h = C(0) - 2\int_{0}^{1}\lambda(t)C'(t\lambda(t))~dt,
 $$
 and 
$$f(x) \le \int_{0}^{x}\lambda(t)C'(t\lambda(t))~dt + C_f,$$
 $$ g(y) \le \int_{0}^{y}\lambda(t)C'(t\lambda(t))~dt + C_g,$$
$$h(z) \le \int_{0}^{z}\lambda(t)C'(t\lambda(t))~dt + C_h.$$
where equality is achieved almost everywhere.
\end{theorem}

\begin{proof}
$\Leftarrow$ Suppose that $\widetilde{f}(x) = \int_{0}^{x}\lambda(t)C'(t\lambda(t))~dt + C_f$, $\widetilde{g}(y) = \int_{0}^{y}\lambda(t)C'(t\lambda(t))~dt + C_g$ and $\widetilde{h}(z) = \int_{0}^{z}t\lambda(t)C'(t\lambda(t))~dt + C_h$. Then the triple $(\widetilde{f}, \widetilde{g}, \widetilde{h})$ is the dual solution by \cref{thm:primal_dual}. Also $f(x) + g(y) + h(z) \le \widetilde{f} + \widetilde{g} + \widetilde{h} \le c(x, y, z)$ and $\int_0^1f(x) + g(x) + h(x)~dx = \int_0^1\widetilde{f} + \widetilde{g} + \widetilde{h}~dx$ so the triple $(f, g, h)$ is the dual solution.

$\Rightarrow$ For any dual solution $(f, g, h)$ there exists a triple $(\widetilde{f}, \widetilde{g}, \widetilde{h})$ such that $f \le \widetilde{f}$, $g \le \widetilde{g}$, $h \le \widetilde{h}$ and $\widetilde{f}(x) = \inf_{y, z}(c(x, y, z) - \widetilde{g}(y) - \widetilde{h}(z))$, $\widetilde{g}(y) = \inf_{x, z}(c(x, y, z) - \widetilde{f}(x) - \widetilde{h}(z))$, $\widetilde{h}(z) = \inf_{x, y}(c(x, y, z) - \widetilde{f}(x) - \widetilde{g}(y))$. One can prove this by applying the Legendre transformation subsequently to $f$, $g$, $h$.

For any $x$, $y$, $z$  inequality $\widetilde{f}(x) + \widetilde{g}(y) + \widetilde{h}(z) \le c(x, y, z)$ holds since $\widetilde{f}(x) = \inf_{y, z}(c(x, y, z) - \widetilde{g}(y) - \widetilde{h}(z))$. Also $$
\int_0^1\widetilde{f}(x)~dx + \int_0^1\widetilde{g}(y)~dy + \int_0^1\widetilde{h}(z)~dz \ge \int_0^1f(x)~dx + \int_0^1 g(y)~dy + \int_0^1 h(z)~dz$$ since $f \le \widetilde{f}$, $g \le \widetilde{g}$ and $h \le \widetilde{h}$. This means that the triple $(\widetilde{f}, \widetilde{g}, \widetilde{h})$ is a dual solution and $\widetilde{f} = f$, $\widetilde{g} = g$, $\widetilde{h} = h$ almost everywhere.

A function $F[y, z]: [0, 1] \to \mathbb{R}$, $F[y, z](x) = c(x, y, z) - \widetilde{g}(y) - \widetilde{h}(z)$ is a Lipschitz continuous function since ${\partial \over \partial x}c(x, y, z)$ is a well-defined continuous function on the cube $[0, 1]^3$. This means that $\widetilde{f}(x)$ is a Lipschitz continuous function since $\widetilde{f}$ is an infimum of the family of Lipschitz continuous functions $F[y, z]$ with common constant $\max_{x, y, z}{\partial \over \partial x}c(x, y, z)$. In particular this means that $\widetilde{f}$ is continuous on the segment $[0, 1]$. Similarly, the functions $\widetilde{g}$ and $\widetilde{h}$ are continuous.

For any primal solution $\mu$ equality $\widetilde{f}(x) + \widetilde{g}(y) + \widetilde{h}(z) = c(x, y, z)$ holds $\mu$-almost everywhere. The set of equality points is closed, because $f$, $g$ and $h$ are continuous. This means that $\widetilde{f}(x) + \widetilde{g}(y) + \widetilde{h}(z) = c(x, y, z)$ on the support of $\mu$. For the primal solution $\mu$ from \cref{section:primal_construction} $\supp(\mu) = M$. So the equality $\widetilde{f}(x) + \widetilde{g}(y) + \widetilde{h}(z) = c(x, y, z)$ holds on the set $M$.

By \cref{thm:weak_uniqueness} the functions $\widetilde{f}$, $\widetilde{g}$ and $\widetilde{h}$ are continuously differentiable and $\widetilde{f}'(x) = \lambda(x)C'(x\lambda(x))$, $\widetilde{g}'(y) = \lambda(y)C'(y\lambda(y))$, $\widetilde{h}'(z) = \lambda(z)C'(z\lambda(z))$. This means that $\widetilde{f}(x) = \hat{f}(x) + C_f$, $\widetilde{g}(y) = \hat{f}(y) + C_g$ and $\widetilde{h}(z) = \hat{f}(z) + C_h$ for some constants $C_f$, $C_g$ and $C_h$. Since $(0, 1, 1) \in M$ the equality holds $C_f + C_g + C_h = c(0, 1, 1) - \hat{f}(0) - \hat{f}(1) - \hat{f}(1) = C(0) - 2\int_{0}^{1}\lambda(t)C'(t\lambda(t))~dt$.
\end{proof}

\section{A priori estimates for the dimension}\label{section:inertion}
Following \cite{brendan_pass} let us introduce the following sets of matrices

$$g_{\{x\}} = g_{\{y, z\}} = 
\begin{pmatrix}
0 & z & y \\
z & 0 & 0 \\
y & 0 & 0
\end{pmatrix},
$$
$$g_{\{y\}} = g_{\{x, z\}} = 
\begin{pmatrix}
0 & z & 0 \\
z & 0 & x \\
0 & x & 0
\end{pmatrix},
$$
$$g_{\{z\}} = g_{\{x, y\}} = 
\begin{pmatrix}
0 & 0 & y \\
0 & 0 & x \\
y & x & 0
\end{pmatrix}.
$$

Further, $G$ is a linear combination of $g_p$ with nonnegative coefficients :

$$ G = \left\{ \left.
\begin{pmatrix}
0 & (\alpha + \beta)z & (\alpha + \gamma)y \\
(\alpha + \beta)z & 0 & (\beta + \gamma)x \\
(\alpha + \gamma)y & (\beta + \gamma)x & 0
\end{pmatrix}
\right| \alpha, \beta, \gamma \ge 0
\right\}.
$$

By Theorem 2.1.2 from \cite{brendan_pass} the supports of solutions to the primal problem are locally contained inside a manifold of dimension 
$$ d = 3 - \text{positive index of inertia of $g$}$$
for any $g \in G$. This index is computed below.

\begin{proposition}
The quadratic form given by 
$$ g = 
\begin{pmatrix}
0 & a & b \\
a & 0 & c \\
b & c & 0
\end{pmatrix}
$$
with non-negative $a$, $b$ and $c$ has positive index of inertia at most $1$.
\end{proposition}
\begin{proof}
Consider two cases. 
\textbf{First case.} Let $a, b, c > 0$. Then principal upper left minors are $\Delta_0 = 1$, $\Delta_1 = 0$, $\Delta_2 = -a^2 < 0$ and $\Delta_3 = 2abc > 0$. So number of sign changes in sequence of principal upper left minors is $2$ and negative index of inertia is $2$. This means that the positive index of inertia is at most $1$.
\textbf{Second case.} Without loss of generality $c = 0$. Then $g$ has the form $2a xy + 2b xz = \frac{1}{2}(x + (ay + bz))^2 - \frac{1}{2}(x - (ay + bz))^2$. Thus the positive index of inertia is at most $1$.
\end{proof}
We see that the local dimension of our solution is indeed not bigger than $2$, but unfortunately this bound does not help to determine the local dimension of our solution without solving problem explicitly.

\section{Extreme points}

We show in this section that the extreme points of the primal solutions are singular to the surface (Hausdorff)
measure on $M$. Applying logarithmic transformation from the proof of \cref{thm:osexists}
and noticing that this is a (locally) bi-Lipschitz transformation
one can easily verify that it is sufficient to prove the claim for the triangle $\Delta$.
Further, projecting $\Delta$ onto the $xy$-hyperplane we reduce the proof of the statement to the proof
of the following fact:

\begin{theorem}
Let $\mu_x, \mu_y$, and $\mu_{x+y}$ be one-dimensional probability measures   on
the axes $x, y$ and on the line $l_{x+y} = \{(x,y) \in \mathbb{R}^2 \colon x=y\}$
respectively. We assume that $\mu_x, \mu_y$ and $\mu_{x+y}$ are compactly supported. Let $\Pi$ be the set of probability measures with projections
$$
\mu_x  = {\rm Pr}_x(\pi), \ \mu_y  = {\rm Pr}_y(\pi), \ \mu_{x+y}  = {\rm Pr}_{x+y}(\pi),
$$
where  ${\rm Pr}_x$, ${\rm Pr}_y$ are projection onto $x, y$, and ${\rm Pr}_{x+y}$
is the projection onto $l_{x+y}$: ${\rm Pr}_{x+y}(x,y) = x + y$.

Assume that $\Pi$ is nonempty and $\pi \in \Pi$ is an extreme point. Then 
there exists a set $S$ of Lebesgue measure zero such that $\pi(S)=1$.
\end{theorem}

\begin{proof}
Without loss of generality let us assume that $\pi$ is supported by $X = [0,1]^2$.
Let us consider the set of tuples of $6$ points
\begin{align*}
N = \Bigl\{ &  \Bigl((x_1, y_2), (x_1, y_3), (x_2, y_1), (x_2, y_3), (x_3, y_2), (x_3, y_1) \Bigr) \colon \ x_1 < x_2 < x_3,
\\&   y_1 < y_2 < y_3, x_1 + y_2 = x_2 + y_1, x_1 + y_3 = x_3 + y_1, x_2 + y_3 = x_3 + y_2
\Bigr\} \subset X^6.
\end{align*}
For arbitrary  $\Gamma \in N$ let us set
$$
 \Gamma_+ = \{(x_1, y_2), (x_2, y_3), (x_3,y_1)\},  \  \Gamma_{-} = \{(x_1, y_3), (x_2, y_1), (x_3,y_2)\}.
$$
Note that $\Gamma = \Gamma_{-} \sqcup \Gamma_{+}$ and uniform distributions on the sets
$\Gamma_+$ and $\Gamma_-$ have the same projections onto the both axes and $l_{x+y}$.

Let us show that there exists a set $S \subset X$ with the properties: $\pi(S)=1$, $S$ does not contain any subset of $6$ points in $N$.
According to a Kellerer's result (see \cite{Kellerer}) the following alternative holds:

\begin{itemize}
\item There exists a measure $\gamma$ on $X^6$ with the property $\gamma(N)>0$, such that 
$\mathrm{Pr}_i \gamma \le \pi$, $1 \le i \le 6$. 
\item For $1 \le i \le 6$ there exists a set $N_i \subset [0,1]^2 = X$ with the property $\pi(N_i)=0$ and
$$
N \subset \cup_{i=1}^6 X \times \ldots \times N_i \times \ldots X. 
$$ 
\end{itemize}
In the second case 
$$
S=X \setminus\cup_{i=1}^6 N_i
$$
will be a desired set. We will prove it later.

First we prove that the first case is impossible. We can assume that $\pi(X^6 \backslash N) = 0$ and $\pi$ is still nonzero.

Suppose that $\Gamma = \Bigl((x_1, y_2), (x_1, y_3), (x_2, y_1), (x_2, y_3), (x_3, y_2), (x_3,y_1) \Bigr)$ is an arbitrary point of $N$ and $B_\Gamma \subset X^6$ is ball with a center at $\Gamma$ and a radius of $\eps < {1 \over 2}\min(x_2 - x_1, x_3 - x_2, y_2 - y_1, y_3 - y_2)$. Also suppose that $\tilde{\gamma} = \gamma|_{B_\Gamma}$ is a (possibly zero) measure on $X^6$ and $\gamma_i = \mathrm{Pr}_i \tilde{\gamma}$ are measures on $X$. If $\gamma(B_\Gamma) >0$ then full measure sets for $\gamma_i$ are pairwise disjoint. In this case measures $\delta_- = {1 \over 3}(\gamma_1 + \gamma_4 + \gamma_6)$ and $\delta_+ = {1 \over 3}(\gamma_2 + \gamma_3 + \gamma_5)$ are distinct and have the same projections onto the axes and diagonal $l_{x + y}$. 

\begin{lemma}
$\delta_- = {1 \over 3}(\gamma_1 + \gamma_4 + \gamma_6)$ and $\delta_+ = {1 \over 3}(\gamma_2 + \gamma_3 + \gamma_5)$ have the same projections onto the axes $x$, $y$ and the line $l_{x+y}$.
\end{lemma}
\begin{proof}
The functions $\Pr_x \circ \Pr_1$ and $\Pr_x \circ \Pr_2$, $\Pr_x \circ \Pr_3$ and $\Pr_x \circ \Pr_4$, $\Pr_x \circ \Pr_5$ and $\Pr_x \circ \Pr_6$,  coincide on $N$. So the images of $\pi$ under this projections coincide. That means $\Pr_x(\gamma_1) = \Pr_x(\gamma_2)$, $\Pr_x(\gamma_3) = \Pr_x(\gamma_4)$, $\Pr_x(\gamma_5) = \Pr_x(\gamma_6)$.

Analogously 
$\Pr_y(\gamma_1) = \Pr_y(\gamma_5)$, 
$\Pr_y(\gamma_2) = \Pr_y(\gamma_4)$, 
$\Pr_y(\gamma_3) = \Pr_y(\gamma_6)$ and 
$\Pr_{x+y}(\gamma_1) = \Pr_{x+y}(\gamma_3)$,
$\Pr_{x+y}(\gamma_2) = \Pr_{x+y}(\gamma_6)$,
$\Pr_{x+y}(\gamma_4) = \Pr_{x+y}(\gamma_5)$.
\end{proof}

Also $\delta_- \le \pi$ and $\delta_+ \le \pi$ since $\gamma_i =\mathrm{Pr}_i \tilde{\gamma} \le \mathrm{Pr}_i \gamma \le \pi$.

Hence $\pi_1 = \pi + \delta_+ - \delta_-$ and $\pi_2 = \pi - \delta_+ + \delta_-$ are nonnegative measures and have the same projections as $\pi$. So $\pi = {1 \over 2}(\pi_1 + \pi_2)$ is not an extreme point.

That means that for any $\Gamma \in N$ the measure of $B_\Gamma$ with respect to $\gamma$ is 0. Hence $\gamma(N) = 0$ which contradicts the assumption.

Thus we get that there exists a set $S$ with $\pi(S) = 1$ such that $S$ does not contain the sets of the type
\begin{align*}
\Bigl\{ &  (x_1, y_2), \ (x_1, y_3), \ (x_2, y_1), \ (x_2, y_3), \ (x_3, y_2), \ (x_3,y_1) , \ \ x_1 < x_2 < x_3, \ y_1 < y_2 < y_3,
\\&  x_1 + y_2 = x_2 + y_1, \ x_1 + y_3 = x_3 + y_1,\ x_2 + y_3 = x_3 + y_2
\Bigr\}.
\end{align*}

Let us show that $S$ has Lebesgue  measure zero. Assuming the contrary, let us apply the Lebesgue's density theorem.
According to this theorem for almost all  $(x, y) \in S$ and every $\varepsilon>0$ there exists a $r$-neighborhood $U$ of $(x, y)$
such that $\lambda(U \cap S) > (1 - \varepsilon) \lambda(U)$.

On the other hand, for all $\alpha$ and $\beta$ the tuple of points  
\begin{align*} \Bigl\{(x+\alpha, y+\beta), {}& (x+\alpha, y+\frac{r}{10}+\beta), (x+\frac{r}{10}+\alpha, y+\beta),
 (x+\frac{r}{10}+\alpha, y+\frac{2r}{10}+\beta), \\& (x+\frac{2r}{10}+\alpha, y+\frac{r}{10}+\beta), (x+\frac{2r}{10}+\alpha, y+\frac{2r}{10}+\beta)
\Bigr\}
\end{align*}
 belongs to $M$. Hence at least one of these points does not belong to  $S$.
 If  $0 \le \alpha, \beta \le \frac{r}{10}$, all these points belong to the  $r$-neighborhood of $(x, y)$,
 hence the measure of the set $U \setminus S)$ is at least  $\frac{r^2}{100}$.
	Choosing  $\varepsilon < \frac{1}{100\pi}$ one gets a contradiction with the Lebesgue's density theorem.
\end{proof}

\begin{remark}
\cref{conj:hausdorff} says that there exist extreme measures with Hausdorff dimension less than $2$.
Numerical experiments reveal certain empirical evidence of this. Nevetheless, we were not able to verify this conjecture. In general, it is not true that sets which do not contain given configurations of points have dimension strictly less than the ambient space (see \cite{Maga, Mathe}).
An example of a low-dimensional solution is given in \cite[Theorem 4.6]{dim-ger-nen}.
\end{remark}

\appendix\section{Discrete case}
Consider the following problem. 

\begin{problem}\label{prob:discrete_monge}
We are given three copies $A, B, C$ of the set $\{1, \ldots, n \}.$ Divide these $3n$ numbers
into $n$ groups of triples $(a, b, c)$, where $a \in A, b \in B, c \in C$.
We want to minimize the sum
$$
S(n) = \sum_{(a, b, c)} abc.
$$
Here the sum is taken over all the triples.
\end{problem}

The main result of this chapter is as follows:
\begin{theorem}\label{thm:discrete}
The minimum $F_D(n)$ of $S(n)$ over all partitions satisfies 
$$F_D(n) \sim C_P n^4,$$ where $C_P$ is the value of the integral in the primal problem.
\end{theorem}

\subsection{Connection with rearrangement inequality}

The rearrangement inequality can be formulated as follows:

\begin{theorem}[Rearrangement inequality]
Assume that $$x_1 \le x_2 \le \dots \le x_n,$$ $$y_1 \le y_2 \le \dots \le y_n$$ are two ordered sets of real numbers, $\sigma$ is a permutation (rearrangement) of $\{1, 2, \dots, n\}$. Then the following inequality holds:
$$
x_1y_1 + x_2y_2 + \dots + x_ny_n \ge x_1y_{\sigma(1)} + x_2y_{\sigma(2)} + \dots + x_ny_{\sigma(n)} \ge x_1y_n + x_2y_{n - 1} + \dots + x_ny_1.
$$

In other words, for the expression $x_1y_{\sigma(1)} + x_2y_{\sigma(2)} + \dots + x_ny_{\sigma(n)}$ the maximum is attained at the identity permutation $\sigma$, and the minimum is attained at the permutation $\begin{pmatrix}
1& 2& \dots& n\\
n& n-1& \dots& 1
\end{pmatrix}$.
\end{theorem}

There exists a generalization of the rearrangement inequality for the case of several sets of variables:

\begin{theorem}\label{thm:generalize_rearrangement}
Assume we are given $s$ ordered sequences
$x_1^{(i)} \le x_2^{(i)} \le \dots \le x_n^{(i)},i=1, \dots, s$. Consider the following functions of permutations 
$$
V(\sigma_1, \dots, \sigma_s) = x_{\sigma_1(1)}^{(1)} x_{\sigma_2(1)}^{(2)} \dots  x_{\sigma_s(1)}^{(n)} + \dots + x_{\sigma_1(n)}^{(1)} x_{\sigma_2(n)}^{(2)} \dots  x_{\sigma_s(n)}^{(n)}.
$$
Let $\sigma_0$ be the identity permutation. Then for any permutation set $\sigma_1, \dots, \sigma_n$ the inequality $V(\sigma_0, \dots, \sigma_0) \ge V(\sigma_1, \sigma_2, \dots, \sigma_s)$ holds.
\end{theorem}

Unfortunately, we do not know for which set of permutations $\sigma_1, \dots, \sigma_s$ the value of $V(\sigma_1, \sigma_2, \dots, \sigma_s)$ is minimal. 

The permutations in generalised rearrangement inequality correspond to a Monge solution for the multimarginal Monge-Kantorovich problem with cost function $x_1x_2\ldots x_s$ and the marginals equal to counting measures on $x^{(i)}_j$. 
We remark that the generalized rearrangement inequality for $3$ variables corresponds to the maximization problem $\int xyz d \pi \to \max$.

\subsection{Approximation of a partition by measures}

Let us introduce some notations. 
For every partition $$Sp = \{(x_i, y_i, z_i) \mid 1 \le i \le n\}$$ 
of $$
A = B = C = \{1, 2, \dots, n\}.
$$ into triples 
define $$S_0(Sp) = \sum_{(x_i, y_i, z_i) \in Sp} x_i y_i z_i.$$ Denote by $|Sp| = n$ the \textit{size} of partition $Sp$.

Let us try to reduce our problem to the transportation problem with the cost function $xyz$. With this purpose we construct the corresponding measure $\mu_1(Sp)$ on $[0, 1]^3$ which is concentrated at points with denominator $n$,
namely every point $({x_i \over n}, {y_i \over n}, {z_i \over n})$ carries the mass ${1 \over n}$. Set $S_1(Sp) = \int_{[0, 1]^3}xyz~d\mu_1(Sp)$. It is easy to check that $$n^4S_1(Sp) = S_0(Sp).$$ 

Projections of $\mu_1(Sp)$ on axes are discrete: measures of points ${i \over n}, 1 \le i \le n$ are equal to ${1 \over n}$. Thus measure $\mu_1(Sp)$ is not $(3, 1)$-stochastic in our sense, since its projections are not Lebesgue measures. This can be easily fixed. To this end, let us introduce another measure $\mu_2(Sp)$ on $[0, 1]^3$: for all $1 \le k \le n$ there exists the uniform measure 
on $$I_k = \left[{x_k - 1 \over n}, {x_k \over n}\right] \times \left[{y_k - 1 \over n}, {y_k \over n}\right] \times \left[{z_k - 1 \over n}, {z_k \over n}\right]$$ 
with density $n^2$ (it is chosen in such a way that the measure of the whole given little cube equals $\frac{1}{n}$). This measure is $(3, 1)$-stochastic.

Set $$S_2(Sp) = \int_{[0, 1]^3}xyz~d\mu_2(Sp).$$ Let us estimate $S_2(Sp)$. For this, we set $$\eps(n) = \sup\left(|x_1y_1z_1 - x_2y_2z_2|, \text{ subject to } \max(|x_1 - x_2|, |y_1 - y_2|, |z_1 - z_2|) \le {1 \over n}\right).$$ Function $xyz$ is continuous on $[0, 1]^3$, then it is uniformly continuous on the given cube. It immediately follows that $\eps(n) \to 0$ for $n \to \infty$. Then we can estimate $|S_1(Sp) - S_2(Sp)|$:
\begin{align*}
|S_1(Sp) - S_2(Sp)| &= \left|\sum_{(x_k, y_k, z_k) \in Sp}\int_{I_k} (xyz - x_ky_kz_k)~d\mu_2(Sp)\right| \\
&\le \sum_{(x_k, y_k, z_k) \in Sp}\int_{I_k} |xyz - x_ky_kz_k|~d\mu_2(Sp) \\
&\le \sum_{(x_k, y_k, z_k) \in Sp}\int_{I_k} \eps(n)~d\mu_2(Sp) \\&= \eps(n) \xrightarrow[n \to \infty]{} 0.
\end{align*}

Thus, $\lim\limits_{n \to \infty}{1 \over n^4}F_D(n)$ exists if and only if there exists $\lim\limits_{n \to \infty}\min\limits_{|Sp| = n}S_2(Sp)$ and in case of existence both limits coincide.

\subsection{Convergence}

In the previous subsection, we realized that it is sufficient to consider the problem of finding a partition $Sp$ that minimizes $S_2(Sp)$. In this section we prove that $\lim\limits_{n \to \infty}\min\limits_{|Sp| = n}S_2(Sp)$ exists. Later we will see that $\lim\limits_{n \to \infty}\min\limits_{|Sp| = n}S_2(Sp) = C_P$, where $C_P$ is the optimal value of the functionals in primal and dual problems.

From definition of $C_P$ the following statement immediately follows:
\begin{proposition}
For every partition $Sp$ there holds an inequality $S_2(Sp) \ge C_P$.
\end{proposition}

Indeed, $S_2(Sp)$ is the integral of $xyz$ by $(3, 1)$-stochastic measure, and $C_P$ is the minimum for all $(3, 1)$-stochastic measures. 

\begin{proposition}
The sequence $s_k = \min\limits_{|Sp| = k}S_2(Sp)$ admits a limit.
\end{proposition}
\begin{proof}
The sequence $s_k$ is bounded below by $C = C_P$. First, we check that $s_{n + k} \le \left(n \over n + k\right)^4s_n + {k \over n + k}$. Indeed, let $Sp_n$ be a partition with $S_2(Sp_n) = s_n$. We construct a partition $Sp_{n + k} = Sp_n \cup \{(i, i, i) \mid n + 1 \le i \le n + k\}$ and verify inequality $S_2(Sp_{n + k}) \le \left(n \over n + k\right)^4s_n + {k \over n + k}$:
\begin{align*}
S_2(Sp_{n + k}) &= \sum_{(x_i, y_i, z_i) \in Sp_n}\int_{z_i - 1 \over n + k}^{z_i \over n + k}\int_{y_i - 1 \over n + k}^{y_i \over n + k}\int_{x_i - 1 \over n + k}^{x_i \over n + k}(n + k)^2xyz~dxdydz \\
&+ \sum_{i = n + 1}^{n + k}\int_{i - 1 \over n + k}^{i \over n + k}\int_{i - 1 \over n + k}^{i \over n + k}\int_{i - 1 \over n + k}^{i \over n + k}(n + k)^2xyz~dxdydz \\
&\le \sum_{(x_i, y_i, z_i) \in Sp_n}\int_{z_i - 1 \over n}^{z_i \over n}\int_{y_i - 1 \over n}^{y_i \over n}\int_{x_i - 1 \over n}^{x_i \over n}(n + k)^2 \left(n \over n + k\right)^6uvw~dudvdw + {k \over n + k} \\
&= \left(n \over n + k\right)^4\sum_{(x_i, y_i, z_i) \in Sp_n}\int_{z_i - 1 \over n}^{z_i \over n}\int_{y_i - 1 \over n}^{y_i \over n}\int_{x_i - 1 \over n}^{x_i \over n}n^2uvw~dudvdw + {k \over n + k} \\
&= \left(n \over n + k\right)^4s_n + {k \over n + k},
\end{align*}
where $u := {n + k \over n}x$, $v := {n + k \over n}y$, $w := {n + k \over n}z$.

We also verify that $s_{nk} \le s_n + \eps(n)$. As in the proof of the previous statement, assume that $Sp_n$ is a partition with $S_2(Sp_n) = s_n$.
We construct another partition $Sp_{nk} = \{(u_{k(i - 1) + j}, v_{k(i - 1) + j}, w_{k(i - 1) + j})\} = \{(k(x_i - 1) + j, k(y_i - 1) + j, k(z_i - 1) + j)\}$, where $1 \le i \le n$, $1 \le j \le k$. It is easy to check that $Sp_{nk}$ is a partition.

We estimate $S_2(Sp_{nk})$. For indices $i$ and $j$ 
\begin{align*}
\int_{I_{k(i - 1) + j}}n^2k^2xyz~dxdydz &\le \int_{z_i - 1 \over n}^{z_i \over n}\int_{y_i - 1 \over n}^{y_i \over n}\int_{x_i - 1 \over n}^{x_i \over n}{n^2 \over k}(xyz + \eps(n))~dxdydz \\
&\le {n^2 \over k} \int_{z_i - 1 \over n}^{z_i \over n}\int_{y_i - 1 \over n}^{y_i \over n}\int_{x_i - 1 \over n}^{x_i \over n}xyz~dxdydz + {1 \over nk}\eps(n).
\end{align*}
From this we get:
\begin{align*}
S_2(Sp_{nk}) &= \sum_{i = 1}^n\sum_{j = 1}^k\int_{I_{k(i - 1) + j}}n^2k^2xyz~dxdydz \\
&\le \sum_{i = 1}^n\sum_{j = 1}^k\left({n^2 \over k} \int_{z_i - 1 \over n}^{z_i \over n}\int_{y_i - 1 \over n}^{y_i \over n}\int_{x_i - 1 \over n}^{x_i \over n}xyz~dxdydz + {1 \over nk}\eps(n)\right) \\
&= \eps(n) + \sum_{i = 1}^n\int_{z_i - 1 \over n}^{z_i \over n}\int_{y_i - 1 \over n}^{y_i \over n}\int_{x_i - 1 \over n}^{x_i \over n}n^2xyz~dxdydz \\
&= s_n + \eps(n).
\end{align*}

From these inequalities we find that for $1 \le i \le n$:
$$
s_{kn + i} \le \left(kn \over kn + i\right)^4s_{kn} + {i \over kn + i} \le s_{kn} + {1 \over k + 1} \le s_n + \eps(n) + {1 \over k + 1}.
$$
As ${1 \over k + 1} \to 0$, we get $s_m \le s_n + 2\eps(n)$ for all sufficiently large $m$.

Set $C_1 = \lim\inf s_n$. We prove that $\lim\limits_{n \to \infty}s_n = C_1$. Indeed, for any $\eps > 0$ there exists such $N$, that $s_N < C_1 + {\eps \over 2}$ and $2\eps(N) < {\eps \over 2}$. Then for all sufficiently large $m$ the inequality $s_m \le s_N + 2\eps(N) < C_1 + \eps$ holds. In addition, for all sufficiently large $m$, inequality $s_m > C_1 - \eps$ holds, otherwise there exists a convergent subsequence, with a limit not greater than $C_1 - \eps$. Thus, $\lim\limits_{n \to \infty}s_n = C_1$, in particular, this sequence is convergent.
\end{proof}

From this statement it follows that it suffices to find partitions $Sp_t$ of an arbitrary size for which $\lim\limits_{t \to \infty}S_2(Sp_t) = C_P$.

\subsection{Discrete measure approximation}

Let $\widetilde{\mu}$ be a measure solving the primal problem. For a given $n$ we define another measure $\widetilde{\mu}_n$. We require that 
$\widetilde{\mu}_n$ is uniform 
on every $$I_{ijk} = \left[{i - 1 \over n}, {i \over n}\right] \times \left[{j - 1 \over n}, {j \over n}\right] \times \left[{k - 1 \over n}, {k \over n}\right],$$
$1 \le i, j, k \le n$ 
and satisfies 
$
\int_{I_{ijk}}1~d\widetilde{\mu} = \int_{I_{ijk}}1~d\widetilde{\mu}_n.
$
The latter quantity will be denoted by $\rho_{ijk}$. The resulting measure will be $(3, 1)$-stochastic.

 Set $c_{ijk} = \min(xyz \mid (x, y, z) \in I_{ijk})$. Then for all $(x, y, z) \in I_{ijk}$ there holds $|c_{ijk} - xyz| < \eps(n)$. Hence, it is possible to estimate $|\int_{[0, 1]^3}xyz~d\widetilde{\mu} - \int_{[0, 1]^3}xyz~d\widetilde{\mu}_n|$:

\begin{align*}
&\Bigg|\int_{[0, 1]^3}xyz~d\widetilde{\mu} - \int_{[0, 1]^3}xyz~d\widetilde{\mu}_n\Bigg| \le \sum_{1 \le i, j, k, \le n}\left|\int_{I_{ijk}}xyz~d\widetilde{\mu} - \int_{I_{ijk}}xyz~d\widetilde{\mu}_n\right|  \\
&\le \sum_{1 \le i, j, k, \le n}\left|\int_{I_{ijk}}c_{ijk}~d\widetilde{\mu} - \int_{I_{ijk}}c_{ijk}~d\widetilde{\mu}_n\right| + \eps(n)\left(\int_{I_{ijk}}1~d\widetilde{\mu} + \int_{I_{ijk}}1~d\widetilde{\mu}_n\right) \\
&= \eps(n)\left(\int_{[0, 1]^3}1~d\widetilde{\mu} + \int_{[0, 1]^3}1~d\widetilde{\mu}_n\right) = 2\eps(n).
\end{align*}

For the following discussion we need the following theorem:

\begin{theorem}[Dirichlet's theorem on the Diophantine approximation]
Assume we are given a set of real numbers $(a_1, a_2, \dots, a_d)$. Then for every $\eps > 0$ there exists a natural number $m$ and integers $b_1, b_2, \dots b_d$ such that $|a_i m - b_i| < \eps$ for all $1 \le i \le d$.
\end{theorem}

Applying this theorem for the set $n\rho_{ijk}$, we find that for any $\eps_1$ there exists a natural $m$, such that $\rho_{ijk} = \frac{t_{ijk} + \eps_{ijk}}{nm}$, where $|\eps_{ijk}| < \eps_1$ and all $t_{ijk}$ are integers. 
We construct the measure $\nu_{n,m}$ as follows : on each cube $I_{ijk}$ we define a uniform measure in such a way that the measure of the whole cube $I_{ijk}$ is equal to ${t_{ijk} \over nm}$.

We verify that this measure is $(3, 1)$-stochastic provided $\eps_1 < {1 \over n^2}$. For this it suffices to verify that the sum of all $t_{ijk} \over nm$ with one argument fixed is equal to ${1 \over n}$. Without loss of generality, we fix $i$. Then ${1 \over n} = \sum_{1 \le j, k \le n}\rho_{ijk} = \sum_{1 \le j, k \le n}{t_{ijk} \over nm} + \sum_{1 \le j, k \le n}{\eps_{ijk} \over nm}$ or $m = \sum_{1 \le j, k \le n}t_{ijk} + \sum_{1 \le j, k \le n}\eps_{ijk}$. All $t_{ijk}$ are natural numbers, and $\left|\sum_{1 \le j, k \le n}\eps_{ijk}\right| \le n^2\eps_1 < 1$, thus $\sum_{1 \le j, k \le n}t_{ijk} = m$, as required.

Estimate the difference $|\int_{[0, 1]^3}xyz~d\widetilde{\mu}_n - \int_{[0, 1]^3}xyz~d\nu_{n,m}|$:
\begin{align*}
&\Bigg|\int_{[0, 1]^3}xyz~d\widetilde{\mu}_n - \int_{[0, 1]^3}xyz~d\nu_{n,m}\Bigg| \le \sum_{1 \le i, j, k, \le n}\left|\int_{I_{ijk}}xyz~d\widetilde{\mu}_n - \int_{I_{ijk}}xyz~d\nu_{n,m}\right| \\
&\le \sum_{1 \le i, j, k, \le n}\left|\int_{I_{ijk}}c_{ijk}~d\widetilde{\mu}_n - \int_{I_{ijk}}c_{ijk}~d\nu_{n,m}\right| + \eps(n)\left(\int_{I_{ijk}}1~d\widetilde{\mu}_n + \int_{I_{ijk}}1~d\nu_{n,m}\right)\\
&= \sum_{1 \le i, j, k \le n}c_{ijk}\left|\rho_{ijk} - {t_{ijk} \over nm}\right| +\eps(n)\left(\int_{[0, 1]^3}1~d\widetilde{\mu}_n + \int_{[0, 1]^3}1~d\nu_{n,m}\right) \\
&\le {n^3 \eps_1 \over nm} + 2\eps(n) \le n^2\eps_1 + 2\eps(n).
\end{align*}

Assume we have found a partition $Sp_{nm}$ and the corresponding $\mu_2(Sp_{nm})$ such that every $I_{ijk}$ contains exactly $t_{ijk}$ small cubes with sides ${1 \over nm}$. Then one can control the difference $|\int_{I}xyz~d\nu_{n,m} - \int_{I}xyz~d\mu_2(Sp_{nm})|$ in the same way as above. One can easily check that the upper bound is $2\eps(n)$, hence $| \int_{I}xyz~d\mu_2(Sp_{nm}) - C_P| \le 6\eps(n) + n^2\eps_1$. This number can be less than any preassigned $\eps$: first we choose $n$, such that $6\eps(n) < \eps / 2$, then choose $\eps_1$, such that $n^2\eps_1 < \eps / 2$.

Thus, to complete the main proof of this section, it is sufficient to show that for given numbers $t_{ijk}, 1 \le i, j, k \le n$ it is always possible to construct a partition with the required property. Namely, using the fact that for a fixed $i$ the sum $\sum_{1\le j, k \le n} t_{ijk}$ is equal to $m$, we build a partition $Sp_{nm} = \{(x_i, y_i, z_i) \mid 1 \le i \le nm\}$ with $$\{x_1, \dots, x_{nm}\} = \{y_1, \dots, y_{nm}\} = \{z_1, \dots, z_{nm}\} = \{1, 2, \dots, nm\}$$ such that for fixed $i, j, k \in \{1, \ldots, n\}$ the number of indices $t$ satisfying $$m(i-1) < x_t \le mi, \ m(j - 1) < y_t \le mj, \ m(k-1) < z_t \le mk $$ equals $t_{ijk}$.

In order to do this, we construct a correspondence between the numbers $1$, \dots, $nm$ and the triples $(i, j, k)$, $1 \le i, j, k \le n$, in such a way that to every index it is assigned exactly one triple, and every triple $(i, j, k)$ corresponds to exactly $t_{ijk}$ indices lying in the half-open interval $\left(m(i-1), mi\right]$. The construction is accomplished step by step. 
The interval $\left(m(i-1), mi\right]$ containing the first $t_{i11}$ numbers corresponds to the triple $(i, 1, 1)$, the following $t_{i12}$ numbers corresponds to the triple $(i, 1, 2)$, and so on.
The last $t_{inn}$ numbers are associated with $(i, n, n)$. This procedure is possible because $\sum_{1 \le j, k \le n} t_{ijk} = m$.

Similarly, we construct the correspondences in the second and third coordinates. As a result, every triple $(i, j, k)$ corresponds to a set of numbers $a_{(i, j, k),1}, \dots, a_{(i, j, k), t_{ijk}}$ from $\left(m(i-1), mi\right]$,
numbers $b_{(i, j, k), 1}, \dots, b_{(i, j, k), t_{ijk}}$ from $\left(m(j-1), mj\right]$, and numbers $c_{(i, j, k), 1}, \dots, c_{(i, j, k), t_{ijk}}$ from $\left(m(k-1), mk\right]$.
Then we set: 
$$Sp_{nm} = \{a_{(i, j, k), t}, b_{(i, j, k), t}, c_{(i, j, k), t}\}, \ 1 \le i, j, k \le n, 1\le t \le t_{ijk}.$$ 
Clearly, this will be a partition of size $nm$, since the values of the numbers $a_{(i, j, k), t}$, $b_{(i, j, k), t}$ and $c_{(i, j, k), t}$ are exactly the set $\{1, \ldots, nm\}$.

This completes the proof of \cref{thm:discrete}.


\begin{thebibliography}{10}
	
	\bibitem{AguehC}
	{\sc M.~Agueh and G.~Carlier}, {\em Barycenters in the {Wasserstein} space},
	SIAM J. Math. Anal., 43 (2011), pp.~904--924,
	\url{https://doi.org/10.1137/100805741}.
	
	\bibitem{bogachev}
	{\sc V.~I. Bogachev and A.~V. Kolesnikov}, {\em The {Monge}--{Kantorovich}
		problem: achievements, connections, and perspectives}, Russian Math. Surveys,
	67 (2012), pp.~785--890,
	\url{https://doi.org/10.1070/RM2012v067n05ABEH004808}.
	
	\bibitem{Carlier}
	{\sc G.~Carlier}, {\em On a class of multidimensional optimal transportation
		problems}, J. Convex Anal., 10 (2003), pp.~517--529.
	
	\bibitem{CarNaz}
	{\sc G.~Carlier and B.~Nazareth}, {\em Optimal transportation for the
		determinant}, ESAIM Control Optim. Calc. Var., 14 (2008), pp.~678--698,
	\url{https://doi.org/10.1051/cocv:2008006}.
	
	\bibitem{CPM}
	{\sc M.~Colombo, L.~De~Pascale, and S.~Di~Marino}, {\em Multimarginal optimal
		transport maps for 1-dimensional repulsive costs}, Canad. J. Math., 67
	(2015), pp.~350--368, \url{https://doi.org/10.4153/CJM-2014-011-x}.
	
	\bibitem{CFK}
	{\sc C.~Cotar, G.~Friesecke, and C.~Klueppelberg}, {\em Density functional
		theory and optimal transportation with {Coulomb} cost}, Comm. Pure Appl.
	Math., 66 (2013), pp.~548--599, \url{https://doi.org/10.1002/cpa.21437}.
	
	\bibitem{dim-ger-nen}
	{\sc S.~Di~Marino, A.~Gerolin, and L.~Nenna}, {\em Optimal transportation
		theory with repulsive costs}, in Topological Optimization and Optimal
	Transport: In the Applied Sciences, Berlin; Boston: De Gruyter, 2017,
	pp.~204--256, \url{https://doi.org/10.1515/9783110430417}.
	
	\bibitem{Friesecke}
	{\sc G.~Friesecke, C.~B. Mendl, B.~Pass, C.~Cotar, and C.~Klüppelberg}, {\em
		{$N$}-density representability and the optimal transport limit of the
		{Hohenberg}-{Kohn} functional}, J. Chem. Phys., 139 (2013),
	\url{https://doi.org/10.1063/1.4821351}.
	
	\bibitem{main_work}
	{\sc N.~A. Gladkov, A.~V. Kolesnikov, and A.~P. Zimin}, {\em On multistochastic
		{Monge}--{Kantorovich} problem, bitwise operations, and fractals}, Calc. Var.
	Partial Differential Equations, 58 (2019),
	\url{https://doi.org/10.1007/s00526-019-1610-4}.
	
	\bibitem{griessler}
	{\sc C.~Griessler}, {\em {$C$}-cyclical monotonicity as a sufficient criterion
		for optimality in the multi-marginal {Monge}--{Kantorovich} problem}, Proc.
	Amer. Math. Soc., 146 (2016), pp.~4735--4740,
	\url{https://doi.org/10.1090/proc/14129}.
	
	\bibitem{Kellerer}
	{\sc H.~G. Kellerer}, {\em Duality theorems for marginal problems}, Z.
	Wahrscheinlichkeitstheorie verw Gebiete, 67 (1984), pp.~399--432,
	\url{https://doi.org/10.1007/BF00532047}.
	
	\bibitem{KP}
	{\sc Y.-H. Kim and B.~Pass}, {\em A general condition for {Monge} solutions in
		the multi-marginal optimal transport problem}, SIAM J. Math. Anal., 46
	(2014), pp.~1538--1550, \url{https://doi.org/10.1137/130930443}.
	
	\bibitem{KL}
	{\sc A.~V. Kolesnikov and N.~Lysenko}, {\em Remarks on mass transportation
		minimizing expectation of a minimum of affine functions}, Theory Stoch.
	Process., 21(37) (2016), pp.~22--28.
	
	\bibitem{Maga}
	{\sc P.~Maga}, {\em Full dimensional sets without given patterns}, Real Anal.
	Exchange, 36 (2010--2011), pp.~79--90,
	\url{https://doi.org/10.14321/realanalexch.36.1.0079}.
	
	\bibitem{Mathe}
	{\sc A.~M\'ath\'e}, {\em Sets of large dimension not containing polynomial
		configurations}, Adv. Math., 316 (2017), pp.~691--709,
	\url{https://doi.org/10.1016/j.aim.2017.01.002}.
	
	\bibitem{Pass11}
	{\sc B.~Pass}, {\em Uniqueness and {Monge} solutions in the multimarginal
		optimal transportation problem}, SIAM J. Math. Anal., 43 (2011),
	pp.~2758--2775, \url{https://doi.org/10.1137/100804917}.
	
	\bibitem{Pass12}
	{\sc B.~Pass}, {\em On the local structure of optimal measures in the
		multimarginal optimal transportation problem}, Calc. Var. Partial
	Differential Equations, 43 (2012), pp.~529--536,
	\url{https://doi.org/10.1007/s00526-011-0421-z}.
	
	\bibitem{brendan_pass}
	{\sc B.~Pass}, {\em Multi-marginal optimal transport: {Theory} and
		applications}, ESAIM Math. Model. Numer. Anal., 49 (2015), pp.~1771--1790,
	\url{https://doi.org/10.1051/m2an/2015020}.
	
	\bibitem{Villani}
	{\sc C.~Villani}, {\em Optimal transport: old and new}, vol.~338, Springer
	Science \& Business Media, 2008,
	\url{https://doi.org/10.1007/978-3-540-71050-9}.
	
	\bibitem{zaev}
	{\sc D.~A. Zaev}, {\em On the {Monge}--{Kantorovich} problem with additional
		linear constraints}, Math. Notes, 98 (2015), pp.~725--741,
	\url{https://doi.org/10.1134/S0001434615110036}.
	
\end{thebibliography}
\end{document}